\documentclass[12pt,amstex]{amsart}
\usepackage{txfonts}

\usepackage{epsfig}
\usepackage{amsmath}
\usepackage{amssymb}
\usepackage{amscd}
\usepackage{graphicx}
\usepackage[all]{xy}
\usepackage{pstricks}
\usepackage{pst-node}
\usepackage[english]{babel}
\usepackage{float}
\usepackage{tikz,pgf}

\usetikzlibrary{matrix}
\newtheorem*{thm*}{Theorem}
\newtheorem{thm}{Theorem}[section]
\newtheorem{lem}[thm]{Lemma}

\newtheorem{prop}[thm]{Proposition}
\newtheorem{ques}[thm]{Question}
\newtheorem{cor}[thm]{Corollary}

\theoremstyle{definition}
\newtheorem{defn}[thm]{Definition}

\theoremstyle{remark}
\newtheorem{rem}[thm]{Remark}

\numberwithin{equation}{section}

\input xy
\xyoption{all}

\def \N {\mathbb N}

\def \Z {\mathbb Z}

\def \F {\mathcal F}
\def \G {\mathcal{G}}

\def \Q {{\bf Q}}
\def \RP {{\bf RP}}

\def \id {{\rm id}}

\def \e {\epsilon}
\def \d {\delta}

\newcommand{\robE}{%
\draw [thick] (-0.5,-0.5) rectangle (0.5,0.5); \draw [ thick] (0,-0.5) -- (0,0.5); \draw [thick, ->] (-0.5,0) -- (0.5,0); }

\newcommand{\robB}{%
\draw [thick] (-0.5,-0.5) rectangle (0.5,0.5); \draw [thick] (0,-0.5) -- (0,0.5); \draw [thick, ->] (-0.5,0) -- (0.5,0); \draw [red] (0.25,-0.5) -- (0.25,0.5); } 
\newcommand{\robD}{%
\draw [thick] (-0.5,-0.5) rectangle (0.5,0.5); \draw [ thick] (0,-0.5) -- (0,0.5); \draw [thick, ->] (-0.5,0) -- (0.5,0); \draw [red, ->] (-0.5,-0.25) -- (0.5,-0.25); } 
\newcommand{\robC}{%
\draw [ thick] (-0.5,-0.5) rectangle (0.5,0.5); \draw [ thick] (0,-0.5) -- (0,0.5); \draw [ thick, ->] (-0.5,0) -- (0.5,0); \draw [ red] (0.25,-0.5) -- (0.25,0.5); \draw [ red, ->] (-0.5,-0.25) -- (0.5,-0.25); }

\newcommand{\robA}{%
\draw [ thick] (-0.5,-0.5) rectangle (0.5,0.5); \draw [ thick,<->] (0,-0.5) -- (0,0.5); \draw [ thick,<->] (-0.5,0) -- (0.5,0); \draw [ red, <->] (-0.25,0.5) -- (-0.25,-0.25) -- (0.5, -0.25); }

\newcommand{\TT}{%
\fill[white!50] rectangle (-0.5,-0.5) rectangle (0.5,0.5); }

\newcommand{\Blanco}{\draw[fill=white] (-0.5,-0.5) rectangle (0.5,0.5); }
\newcommand{\Negro}{\draw[fill=black] (-0.5,-0.5) rectangle (0.5,0.5); }

\begin{document}

\title{Dynamical cubes and a criteria for systems having product extensions}
\author{Sebasti\'an Donoso and Wenbo Sun}

 \address{Centro de Modelamiento Matem\'atico and Departamento de Ingenier\'{\i}a
Matem\'atica, Universidad de Chile, Av. Blanco Encalada 2120,
Santiago, Chile \newline  Universit\'e Paris-Est, Laboratoire d'analyse et de math\'ematiques
appliqu\'ees, 5 bd Descartes, 77454 Marne la Vall\'ee
Cedex 2, France} \email{sdonoso@dim.uchile.cl, sebastian.donoso@univ-paris-est.fr }

\address{Department of Mathematics, Northwestern University, 2033
Sheridan Road Evanston, IL 60208-2730, USA}
 \email{swenbo@math.northwestern.edu}
 
\thanks{The first author was supported by CONICYT doctoral fellowship and University of Chile BECI grant.  The
second author was partially supported by NSF grant 1200971.}

\maketitle

\begin{abstract}
For minimal $\mathbb{Z}^{2}$-topological dynamical systems, we introduce a cube structure and a variation of the regionally proximal relation for $\Z^2$ actions, which allow us to characterize product systems and their factors. We also introduce the concept of topological magic systems, which is the topological counterpart of measure theoretic magic systems introduced by Host in his study of multiple averages for commuting transformations. Roughly speaking, magic systems have a less intricate dynamic and we show that every minimal $\Z^2$ dynamical system has a magic extension. We give various applications of these structures, including the construction of some special factors in topological dynamics of $\Z^2$ actions, and a computation of the automorphism group of the minimal Robinson tiling.
\end{abstract}

\section{Introduction}

We start by reviewing the motivation for characterizing cube structures for systems with a single transformation, which was first developed for ergodic measure preserving systems. To show the convergence of some multiple ergodic averages, Host and Kra \cite{HK05} introduced for each $d\in \N$ a factor $\mathcal{Z}_d$ which characterizes the behaviour of those averages. They proved that this factor can be endowed with a structure of a nilmanifold: it is measurably isomorphic to an inverse limit of ergodic rotations on nilmanifolds. To build such a structure, they introduced cube structures over the set of measurable functions of X to itself and they studied their properties. Later, Host, Kra and Maass \cite{HKM} introduced these cube structures into topological dynamics. For $(X,T)$ a minimal dynamical system and for $d\in \N$, they introduced the space of cubes $\Q^{[d+1]}$ which characterizes being topologically isomorphic to an inverse limit of minimal rotations on nilmanifolds. They also defined the $d$-step regionally proximal relation, denoted by $\RP^{[d]}$ which allows one to build the maximal nilfactor. They showed that $\RP^{[d]}$ is an equivalence relation in the distal setting. Recently, Shao and Ye \cite{SY} proved that $\RP^{[d]}$ is an equivalence relation in any minimal system and the quotient by this relation is the maximal nilfactor of order $d$. This theory is important in studying the structure of $\Z$-topological dynamical systems and recent applications of it can be found in \cite{D}, \cite{HSY}, \cite{HSY2}.

Back to ergodic theory, a natural generalization of the averages considered by Host and Kra \cite{HK05} are averages arise from a measurable preserving system of commuting transformations $(X,\mathcal{B},\mu,T_1,\ldots,T_d)$. The convergence of these averages was first proved by Tao \cite{Tao} with further insight given by Towsner \cite{Tow}, Austin \cite{Austin} and Host \cite{H}. We focus our attention on Host's proof. In order to prove the convergence of the averages, Host built an extension of $X$ (\emph{magic} in his terminology) with suitable properties. In this extension he found a characteristic factor that looks like the Cartesian product of single transformations. Again, to build these objects, cubes structures are introduced, analogous to the ones in \cite{HK05}.

\subsection{Criteria for systems having a product extension}
A {\it system with commuting transformations} $(X,S,T)$ is a compact metric space $X$ endowed with two commuting homeomorphisms $S$ and $T$. A {\it product system} is a system of commuting transformations of  the form $(Y\times W, \sigma\times \id, \id\times \tau)$, where $\sigma$ and $\tau$ are homeomorphisms of $Y$ and $W$ respectively (we also say that $(Y\times W, \sigma\times \id, \id\times \tau)$ is the {\it product} of $(Y,\sigma)$ and $(W,\tau)$). These are the simplest systems of commuting transformations one can imagine.

We are interested in understanding how ``far" a system with commuting transformations is from being a product system, and more generally, from being a factor of a product system. To address this question we need to develop a new theory of cube structures for this kind of actions which is motivated by Host's work in ergodic theory and that results in a fundamental tool.

\bigskip

 Let $(X,S,T)$ be a system with commuting transformations $S$ and $T$. The {\it space of cubes} $\Q_{S,T}(X)$ of $(X,S,T)$ is the closure in $X^4$ of the points $(x, S^nx,T^mx,S^nT^mx)$, where $x\in X$ and $n,m\in \Z$.

Our main result is that this structure allows us to characterize systems with a product extension:

\begin{thm} \label{R-ProductExtension}
 Let $(X,S,T)$ be a minimal system with commuting transformations $S$ and $T$. The following are equivalent:
   \begin{enumerate}
    \item $(X,S,T)$ is a factor of a product system;
    \item If ${\bf x}$ and ${\bf y} \in \Q_{S,T}(X)$ have three coordinates in common, then ${\bf x}={\bf y}$;
     \item If $(x,y,a,a)\in \Q_{S,T}(X)$ for some $a\in X$, then $x=y$; 
     \item If $(x,b,y,b)\in \Q_{S,T}(X)$ for some $b\in X$, then $x=y$; 
     \item If $(x,y,a,a) \in \Q_{S,T}(X)$ and $(x,b,y,b)\in \Q_{S,T}(X)$ for some $a,b\in X$, then $x=y$.
   \end{enumerate}
\end{thm}

\bigskip

The cube structure $\Q_{S,T}(X)$ also provides us a framework for studying the structure of a system with commuting transformations. We introduce the {\it $(S,T)$-regionally proximal relation} $\mathcal{R}_{S,T}(X)$ of $(X,S,T)$, defined as
$$\mathcal{R}_{S,T}(X)\coloneqq \{(x,y)\colon (x,y,a,a), (x,b,y,b) \in \Q_{S,T}(X) \text{ for some } a,b \in X\}.$$

We remark that in the case $S=T$, this definition coincides with $\Q^{[2]}$ and $\RP^{[1]}$ defined in \cite{HKM}. When $S\neq T$, the relation $\mathcal{R}_{S,T}(X)$ is included in the regionally proximal relation for $\Z^2$ actions \cite{Aus} but can be different. So $\mathcal{R}_{S,T}(X)$ is a variation of $\RP^{[1]}$ for $\Z^2$ actions.

In a distal system with commuting transformations, it turns out that we can further describe properties of $\mathcal{R}_{S,T}(X)$.  We prove that $\mathcal{R}_{S,T}(X)$ is an equivalence relation and the quotient of $X$ by this relation defines the {\it maximal factor with a product extension} (see Section 5 for definitions).

We also study the topological counterpart of the ``magic extension" in Host's work \cite{H}. We define the magic extension in the topological setting and show that in this setting, every minimal system with commuting transformations admits a minimal magic extension (Proposition \ref{MagicExtension}). Combining this with the properties of the cube $\Q_{S,T}(X)$ and the relation $\mathcal{R}_{S,T}(X)$, we are able to prove Theorem \ref{R-ProductExtension}.

We provide several applications, both in a theoretical framework and to real systems. Using the cube structure, we study some representative tiling systems. For example, we show that the $\mathcal{R}_{S,T}$ relation on the two dimensional Morse tiling system is trivial. Therefore, it follows from Theorem \ref{R-ProductExtension} that it has a product extension. Another example we study is the minimal Robinson tiling system. Since automorphisms preserve the $\mathcal{R}_{S,T}$ relation, we can study the automorphism group of a tiling system by computing its $\mathcal{R}_{S,T}(X)$ relation. We show that the automorphism group of the minimal Robinson tiling system consists of only the $\Z^{2}$-shifts.

Another application of the cube structure is to study the properties of a system having a product system as an extension, which include (see Section 6 for definitions):

\begin{enumerate}
    \item Enveloping semigroup: we show that $(X,S,T)$ has a product extension if and only if the enveloping semigroup of $X$ is automorphic.
    \item Disjoint orthogonal complement: we show that if $(X,S,T)$ is an $S$-$T$ almost periodic system with a product extension, then $(X,S,T)$ is disjoint from $(Y,S,T)$ if and only if both $(Y,S)$ and $(Y,T)$ are weakly mixing system.
     \item Set of return times: we show that in the distal setting, $(x,y)\in\mathcal{R}_{S,T}(X)$ if and only if the set of return time of $x$ to any neighborhood of $y$ is an $\mathcal{B}_{S,T}^{*}$ set.
     \item Topological complexity: we define the topological complexity of a system and show that in the distal setting, $(X,S,T)$ has a product extension if and only if it has bounded topological complexity.
\end{enumerate}

\subsection{Organization of the paper}

In Section 3, we formally define the cube structure, the $(S,T)$-regionally proximal relation and the magic extension in the setting of systems with commuting transformations. We prove that every minimal system with commuting transformations has a minimal magic extension, and then we use this to give a criteria for systems having a product extension (Theorem \ref{R-ProductExtension}).

In Section 4, we compute the $\mathcal{R}_{S,T}(X)$ relation for some tiling systems and provide some applications.

In Section 5, we study further properties of the $\mathcal{R}_{S,T}(X)$ relation in the distal case.
In Section 6, we study various properties of systems with product extensions, which includes the study of its
enveloping semigroup, disjoint orthogonal complement, set of return times, and topological complexity.

\section{Notation}

A {\it topological dynamical system} is a pair $(X,G_{0})$, where $X$ is a compact metric space and $G_{0}$ is a group of homeomorphisms of the space $X$ into itself. We also refer to $X$ as a $G_{0}$-dynamical system. We always use $d(\cdot,\cdot)$ to denote the metric in $X$ and we let $\Delta_X\coloneqq \{(x,x)\colon x\in X\}$ denote the diagonal of $X\times X$.

If $T\colon X\to X$ is a homeomorphism of the space $X$, we use $(X,T)$ to denote the topological dynamical system $(X,\{T^n\colon n\in \Z\})$.

  If $S\colon X \to X$ and $T\colon X\to X$ are two commuting homeomorphisms of $X$,  we write $(X,S,T)$ to denote the topological dynamical system $(X,\{S^nT^m\colon n,m\in \Z \})$. The transformations $S$ and $T$ span a $\mathbb{Z}^2$-action, but we stress that we consider this action with a given pair of generators. Throughout this paper, we always use $G\cong\mathbb{Z}^2$ to denote the group generated by $S$ and $T$.

A {\it factor map} between the dynamical systems $(Y,G_{0})$ and $(X,G_{0})$ is an onto, continuous map $\pi\colon Y\to X$ such that $\pi \circ g=g\circ \pi$ for every $g\in G_{0}$. We say that $(Y,G_{0})$ is an {\it extension} of $(X,G_{0})$ or that $(X,G_{0})$ is a {\it factor} of $(Y,G_{0})$. When $\pi$ is bijective, we say that $\pi$ is an isomorphism and that $(Y,G_0)$ and $(X,G_0)$ are isomorphic. By a factor map between two systems $(Y,S,T)$ and $(X,S',T')$ of commuting transformations, we mean that $\pi\circ S =S'\circ\pi$ and $\pi\circ T=T'\circ\pi$.

We say that $(X,G_{0})$ is {\it transitive} if there exists a point in $X$ whose orbit is dense. Equivalently, $(X,G_0)$ is transitive if for any two non-empty open sets $U,V\subseteq X$ there exists $g\in G_0$ such that $U\cap g^{-1} V\neq \emptyset$.

A system $(X,G_{0})$ is {\it weakly mixing} if the Cartesian product $X\times X$ is transitive under the action of the diagonal of $G_{0}$. Equivalently, $(X,G_0)$ is weakly mixing if for any four non-empty open sets $A,B,C,D\subseteq X$ there exists $g\in G_0$ such that $A\cap g^{-1}B\neq \emptyset$ and $C\cap g^{-1}D\neq \emptyset$.

Let $(X,G_{0})$ be a topological dynamical system. Denote the orbit closure of a point $x\in X$ by $\mathcal{O}_{G_{0}}(x)\coloneqq\overline{\{gx\colon g\in G_{0}\}}$. If $G_{0}$ is generated by one element $R$, we write $\mathcal{O}_R(x)\coloneqq \mathcal{O}_{G_0}(x)$ for convenience.
We say that $(X,G_{0})$ is {\it minimal} if the orbit of any point is dense in $X$. A system $(X,G_{0})$ is {\it pointwise almost periodic} if for any $x\in X$, the dynamical system $(\mathcal{O}_{G_{0}}(x),G_{0})$ is minimal.


\section{Cube structures and general properties}

\subsection{Cube structures and the $(S,T)$-regionally proximal relation}

\begin{defn}
   For a system $(X,S,T)$ with commuting transformations $S$ and $T$, let $\mathcal{F}_{S,T}$ denote the subgroup of $G^{4}$ generated by $\id\times S\times \id \times S$ and $\id\times\id\times T \times T$ (recall that $G$ is the group spanned by $S$ and $T$). Write $G^{\Delta}\coloneqq\{g\times g\times g\times g\in G^{4}\colon g\in G\}$. Let $\mathcal{G}_{S,T}$ denote the subgroup of $G^{4}$ generated by $\mathcal{F}_{S,T}$ and $G^{\Delta}$.
\end{defn}

The main structure studied in this paper is a notion of cubes for a system with commuting transformations:

\begin{defn}
Let $(X,S,T)$ be a system with commuting transformations $S$ and $T$. We define
\begin{equation}\nonumber
  \begin{split}
  &\bold{Q}_{S,T}(X)=\overline{\{(x,S^nx,T^mx,S^nT^mx)\colon x\in X, n,m\in \Z\}};
  \\&\bold{Q}_{S}(X)=\pi_{0}\times\pi_{1}(\bold{Q}_{S,T}(X))=\overline{\{(x,S^{n}x)\in X\colon x\in X, n\in\mathbb{Z}\}};
  \\&\bold{Q}_{T}(X)=\pi_{0}\times\pi_{2}(\bold{Q}_{S,T}(X))=\overline{\{(x,T^{n}x)\in X\colon x\in X, n\in\mathbb{Z}\}};
    \\&\bold{K}^{x_{0}}_{S,T}=\overline{\{(S^{n}x_{0},T^{m}x_{0},S^{n}T^{m}x_{0})\in X^3\colon n,m\in\mathbb{Z}\}} \text{ for all }x_0 \in X,
  \end{split}
\end{equation}
where $\pi_{i}\colon X^{4}\to X$ is the projection to the $i$-th coordinate in $X^4$ for $i=0,1,2,3$.

\end{defn}

We start with some basic properties of $\Q_{S,T}(X)$. The following proposition follows immediately from the definitions:

\begin{prop}\label{sym} Let $(X,S,T)$ be a minimal system with commuting transformations $S$ and $T$. Then,
  \begin{enumerate}
   \item $(x,x,x,x) \in \Q_{S,T}(X)$ for every $x\in X$;
   \item $\Q_{S,T}(X)$ is invariant under $\mathcal{G}_{S,T}$;
   \item If $(x_0,x_1,x_2,x_3)\in \Q_{S,T}(X)$, then $(x_2,x_3,x_0,x_1),(x_1,x_0,x_3,x_2)\in \Q_{S,T}(X)$ and \newline  $(x_0,x_2,x_1,x_3)\in \Q_{T,S}(X)$;
   \item If $(x_0,x_1,x_2,x_3)\in \Q_{S,T}(X)$, then $(x_0,x_1),(x_2,x_3)\in \Q_{S}(X)$ and $(x_0,x_2),(x_1,x_3)\in \Q_{T}(X)$;
   \item If $(x_0,x_1)\in \Q_{S}(X)$, then $(x_0,x_1,x_0,x_1)\in \Q_{S,T}(X)$; If $(x_0,x_1)\in \Q_{T}(X)$, then $(x_0,x_0,x_1,x_1)\in \Q_{S,T}(X)$;
   \item $(x,y)\in \Q_{R}(X)\Leftrightarrow (y,x)\in \Q_{R}(X)$ for all $x,y\in X$, where $R$ is either $S$ or is $T$.
  \end{enumerate}
\end{prop}

It is easy to see that $(\Q_{S,T}(X), \mathcal{G}_{S,T})$ is a topological dynamical system. Moreover, we have:

\begin{prop} \label{Q_STMinimal}
 Let $(X,S,T)$ be a minimal system with commuting transformations $S$ and $T$. Then $(\Q_{S,T}(X), \mathcal{G}_{S,T})$ is a minimal system. Particularly, taking $R$ to be either $S$ or $T$, $\Q_R(X)$ is minimal under the action generated by $\id\times R$ and $g\times g$ for $g \in G$.
\end{prop}

 \begin{proof} We use results on the enveloping semigroups, and defer the definitions and basic properties to Appendix A.

 The proof is similar to the one given in page 46 of \cite{Glasner} for some similar diagonal actions. Let $E(\Q_{S,T}(X),\mathcal{G}_{S,T})$ be the enveloping semigroup of $(\Q_{S,T}(X),\mathcal{G}_{S,T})$. For $i=0,1,2,3$, let $\pi_i\colon\Q_{S,T}(X)\to X$ be the projection onto the $i$-th coordinate and let $\pi_i^{\ast}\colon E(\Q_{S,T}(X),\mathcal{G}_{S,T})\to E(X,G)$ be the induced factor map.

 Let $u\in E(\Q_{S,T}(X),G^{\Delta})$ denote a minimal idempotent. We show that $u$ is also a minimal idempotent in $E(\Q_{S,T}(X),\mathcal{G}_{S,T})$. By Theorem \ref{Enveloping1}, it suffices to show that if $v\in  E(\Q_{S,T}(X),\mathcal{G}_{S,T})$ with $vu=v$, then $uv=u$. Projecting onto the corresponding coordinates, we deduce that $\pi^{\ast}_i(vu)=\pi_{i}^{\ast}(v)\pi_i^{\ast}(u)=\pi_i^{\ast}(v)$ for $i=0,1,2,3$. It is clear that the projection of a minimal idempotent to $E(\Q_{S,T}(X),G^{\Delta})$ is a minimal idempotent in $E(X,G)$. Since $\pi_{i}^{\ast}(v)\pi_{i}^{\ast}(u)=\pi_{i}^{\ast}(v)$, by Theorem \ref{Enveloping1} we deduce that $\pi_{i}^{\ast}(u)\pi_{i}^{\ast}(v)=\pi_{i}^{\ast}(u)$ for $i=0,1,2,3$. Since the projections onto the coordinates determine
an element of $E(\Q_{S,T}(X),\mathcal{G}_{S,T})$, we have that $uv=u$. Thus we conclude that $u$ is a minimal idempotent in $E(\Q_{S,T}(X),\mathcal{G}_{S,T})$.

 For any $x\in X$, $(x,x,x,x)$ is a minimal point under $G^{\Delta}$. So there exists a minimal idempotent $u\in E(\Q_{S,T}(X),G^{\Delta})$ such that $u(x,x,x,x)=(x,x,x,x)$. Since $u$ is also a minimal idempotent in $E(\Q_{S,T}(X),\mathcal{G}_{S,T})$, the point $(x,x,x,x)$ is minimal under $\mathcal{G}_{S,T}$. Since the orbit closure of $(x,x,x,x)$ under $\mathcal{G}_{S,T}$ is $\Q_{S,T}(X)$, we have that $(\Q_{S,T}(X),\mathcal{G}_{S,T})$ is a minimal system.

 The fact that $\Q_R(X)$ is minimal follows immediately by taking projections.
\end{proof}

We remark that $\bold{K}^{x_{0}}_{S,T}$ is invariant under $\widehat{S}\coloneqq S\times \id \times S$ and $\widehat{T}\coloneqq \id \times T \times T$. We let $\mathcal{F}_{S,T}^{x_0}$ denote the action spanned by $\widehat{S}$ and $\widehat{T}$. We note that $(\bold{K}^{x_{0}}_{S,T},\mathcal{F}_{S,T}^{x_0})$ is not necessarily minimal, even if $X$ is minimal (the minimality of $\bold{K}^{x_{0}}_{S,T}$ implies the minimality of $\mathcal{O}_S(x_0)$ under $S$ and the minimality of $\mathcal{O}_T(x_0)$ under $T$, which does not always hold). See the examples in Section 4.

The following lemma follows from the definitions:
\begin{lem}
 Let $\pi\colon Y\to X$ be a factor map between two minimal systems $(Y,S,T)$ and $(X,S,T)$ with commuting transformations $S$ and $T$. Then $\pi\times\pi\times\pi\times \pi(\Q_{S,T}(Y))=\Q_{S,T}(X)$. Therefore, $\pi\times \pi(\Q_{S}(Y))=\Q_S(X)$ and $\pi\times \pi(\Q_{T}(Y))=\Q_T(X)$.
\end{lem}

Associated to the cube structure, we define a relation in $X$ as was done in ~\cite{HKM} with cubes associated to a $\mathbb{Z}$-system.
This is the main relation we study in this paper:

\begin{defn} Let $(X,S,T)$ be a minimal system with commuting transformations $S$ and $T$. We define
\begin{equation}\nonumber
  \begin{split}
  &\mathcal{R}_{S}(X)=\{(x,y)\in X\times X\colon (x,y,a,a)\in\bold{Q}_{S,T}(X) \text{ for some } a\in X\};
  \\&\mathcal{R}_{T}(X)=\{(x,y)\in X\times X\colon (x,b,y,b)\in\bold{Q}_{S,T}(X) \text{ for some } b\in X\};
  \\&\mathcal{R}_{S,T}(X)=\mathcal{R}_{S}(X)\cap\mathcal{R}_{T}(X).
  \end{split}
\end{equation}
\end{defn}

It then follows from (3) of Proposition \ref{sym} that $\mathcal{R}_{S}(X),\mathcal{R}_{T}(X),\mathcal{R}_{S,T}(X)$ are symmetric relations, i.e. $(x,y)\in A$ if and only if $(y,x)\in A$ for all $x,y\in X$, where $A$ is $\mathcal{R}_{S}(X),\mathcal{R}_{T}(X)$ or $\mathcal{R}_{S,T}(X)$.
It is worth noting that in the case $S=T$, $\mathcal{R}_{S,T}(X)$ is the regionally proximal relation $\bold{RP}^{[1]}$ defined in ~\cite{HKM}.

Using these definitions, our main Theorem \ref{R-ProductExtension} can be rephrased as (we postpone the proof to Section 3.4):

\begin{thm*} 
 Let $(X,S,T)$ be a minimal system with commuting transformations $S$ and $T$. The following are equivalent:
   \begin{enumerate}
    \item $(X,S,T)$ is a factor of a product system;
    \item If ${\bf x}$ and ${\bf y} \in \Q_{S,T}(X)$ have three coordinates in common, then ${\bf x}={\bf y}$;
     \item $\mathcal{R}_{S}(X)=\Delta_{X}$;
     \item $\mathcal{R}_{T}(X)=\Delta_{X}$;
     \item $\mathcal{R}_{S,T}(X)=\Delta_{X}$.
   \end{enumerate}
\end{thm*}

\begin{rem}
  In the case where $(X,S,T)=(Y\times W,\sigma\times \id,\id\times \tau)$ is exactly a product system, we have that $$\Q_{S,T}(X)=\left \{(y_1,w_1),(y_2,w_1),(y_1,w_2),(y_2,w_2)\colon ~ y_1,y_2 \in Y, ~ w_1,w_2 \in W \right \}.$$

  In this case, $\mathcal{R}_{S,T}(X)=\Delta_{X}$ holds for trivial reasons. Suppose that $((y_{1},w_{1}),(y_{2},w_{2}))\in\mathcal{R}_{S,T}(X)$ for some $(y_{1},w_{1})$, $(y_{2},w_{2})\in X$. Since $((y_{1},w_{1}),(y_{2},w_{2}))\in\mathcal{R}_{S}(X)$, there exists $a\in X$ such that $((y_{1},w_{1}),(y_{2},w_{2}),a,a)\in\Q_{S}(X)$. Therefore $w_2=w_1$ and $(y_1,w_2)=a=(y_2,w_2)$, which implies that $y_1=y_2$. Thus $\mathcal{R}_{S,T}(X)=\Delta_{X}$.
\end{rem}

\subsection{Magic systems}
We construct an extension of a system with commuting transformations which behaves like a product system
for use in the sequel.
Following the terminology introduced in \cite{H} in the ergodic setting, we introduce the notion of a magic system in the topological setting:

\begin{defn}
   A minimal system $(X,S,T)$ with commuting transformations $S$ and $T$ is called a {\it magic system} if $\mathcal{R}_{S}(X)\cap \mathcal{R}_{T}(X)=\Q_S(X)\cap\Q_{T}(X)$.
\end{defn}

We remark that the inclusion in one direction always holds:

\begin{lem}
  Let $(X,S,T)$ be a system with commuting transformations $S$ and $T$. Then $\mathcal{R}_S(X)\cap \mathcal{R}_T(X) \subseteq \Q_S(X)\cap \Q_T(X)$.
\end{lem}
\begin{proof}
  Suppose $(x,y)\in\mathcal{R}_S(X)\cap \mathcal{R}_T(X)$. Then in particular $(x,y)\in\mathcal{R}_S(X)$. So there exists $a\in X$ such that $(x,y,a,a)\in \Q_{S,T}(X)$. Taking the projections onto the first two coordinates, we have that $(x,y)\in \Q_{S}(X)$. Similarly, $(x,y)\in \Q_{T}(X)$, and so $\mathcal{R}_S(X)\cap \mathcal{R}_T(X) \subseteq \Q_S(X)\cap \Q_T(X)$.
\end{proof}

 In general, not every system with commuting transformations is magic. In fact, $\mathcal{R}_S(X)\cap \mathcal{R}_T(X)$ and $\Q_S(X)\cap \Q_T(X)$ may be very different. For example, let $(\mathbb{T}=\mathbb{R}/\mathbb{Z},T)$ be a rotation on the circle given by $Tx=x+\alpha \mod 1$ for all $x\in\mathbb{T}$, where $\alpha$ is an irrational number. Then $\Q_T(\mathbb{T})\cap \Q_T(\mathbb{T})=\mathbb{T}\times \mathbb{T}$. But $\mathcal{R}_T(\mathbb{T})\cap \mathcal{R}_T(\mathbb{T})=\{(x,x)\in \mathbb{T}^{2}\colon x\in \mathbb{T}\}$ (here we take $S=T$).
  However, we can always regard a minimal system with commuting transformations as a factor of a magic system:

\begin{prop}[Magic extension] \label{MagicExtension}
 Let $(X,S,T)$ be a minimal system with commuting transformations $S$ and $T$. Then $(X,S,T)$ admits a minimal magic extension, meaning it has an extension which is a minimal magic system.
\end{prop}

\begin{proof}
By Section 4 of \cite{Glasner2}, we can find a point $x_{0}\in X$ such that $\Q_S[x_0]\coloneqq \{x\in X\colon (x_0,x)\in \Q_S(X)\}$ and $\Q_T[x_0]\coloneqq \{x\in X\colon (x_0,x)\in \Q_T(X)\}$ coincide with $\mathcal{O}_S(x_0)$ and $\mathcal{O}_T(x_0)$ respectively (moreover, the set of such points is a $G_{\delta}$ set).

  Let $Y$ be a minimal subsystem of the system $(\bold{K}^{x_0}_{S,T},\widehat{S},\widehat{T})$, where $\widehat{S}=S\times \id\times S$, $\widehat{T}=\id\times T\times T$. Since the projection onto the last coordinate defines a factor map from $(Y,\widehat{S},\widehat{T})$ to $(X,S,T)$, there exists a minimal point of $Y$ of the form $\vec{z}=(z_{1},z_{2},x_0)$. Hence, $Y$ is the orbit closure of $(z_1,z_2,x_0)$ under $\widehat{S}$ and $\widehat{T}$. We claim that $(Y,\widehat{S},\widehat{T})$ is a magic extension of $(X,S,T)$.

  It suffices to show that for any $\vec{x}=(x_{1},x_{2},x_{3}),\vec{y}=(y_{1},y_{2},y_{3})\in Y$,
  $(\vec{x},\vec{y})\in \Q_{\widehat{S}}(Y)\cap\Q_{\widehat{T}}(Y)$ implies that
  $(\vec{x},\vec{y})\in \mathcal{R}_{\widehat{S}}(Y)\cap\mathcal{R}_{\widehat{T}}(Y)$.
  Since $(\vec{x},\vec{y})\in \Q_{\widehat{S}}(Y)$ and the second coordinate of $Y$ is invariant under $\widehat{S}$, we get that $x_{2}=y_{2}$. Similarly, $(\vec{x},\vec{y})\in \Q_{\widehat{T}}(Y)$ implies that $x_{1}=y_{1}$.

  We recall that $d(\cdot,\cdot)$ is a metric in $X$ defining its topology. Let $\e>0$.
  Since $(\vec{x},\vec{y}) \in \Q_{\widehat{S}}(Y)$, there exists $\vec{x'}=(x'_{1},x'_{2},x'_{3})\in Y$ and $ n_{0}\in\mathbb{Z}$ such that
  $d(x_{i},x'_{i}) <\e$ for $i=1,2,3$ and that $d(S^{n_{0}}x'_{1},x_{1})<\e$, $d(S^{n_{0}}x'_{3},y_{3})<\e$. Let $0<\delta<\epsilon$ be such that if $x,y\in X$ and $d(x,y)<\delta$, then $d(S^{n_{0}}x,S^{n_{0}}y)<\epsilon$.

  Since $\vec{x'}\in Y$, there exist $n,m\in\mathbb{Z}$ such that $d(x'_{1},S^{n}z_{1})$, $d(x'_{2},T^{m}z_{2})$, $d(x'_{3},S^{n}T^{m}x_0)<\d$. Then $d(S^{n_{0}}x'_{1},S^{n_{0}+n}z_{1})$, $d(S^{n_{0}}x'_{3},S^{n_{0}+n}T^{m}x_0)<\e$.

  Let $0<\d'<\d$ be such that if $x,y\in X$ and $d(x,y)<\d'$, then $d(S^{n}x,S^{n}y)<\d$. Since $\vec{z}\in \bold{K}^{x_{0}}_{S,T}$, we have that $(z_{1},x_0)\in\Q_{T}[x_0]$. By assumption, there exists $m_{0}\in\mathbb{Z}$ such that $d(T^{m_{0}}x_0,z_{1})<\d'$. Then $d(S^{n}T^{m_{0}}x_0,S^{n}z_{1})<\d$ and $d(S^{n+n_{0}}T^{m_{0}}x_0,S^{n+n_{0}}z_{1})<\e$.

  Denote $\vec{z'}=(S^{n}z_{1},T^{m}z_{2},S^{n}T^{m}x_0)\in Y$. Then the distance between
  \begin{equation}\nonumber
  \begin{split}
  (\vec{z'},\widehat{S}^{n_{0}}\vec{z'},\widehat{T}^{m_{0}-m}\vec{z'},\widehat{S}^{n_{0}}\widehat{T}^{m_{0}-m}\vec{z'})
  \end{split}
  \end{equation}
  and the corresponding coordinates of
   $w=(\vec{x},\vec{y},\vec{u},\vec{u})$ is smaller than $C\epsilon$ for some uniform constant $C>0$, where $\vec{u}=(x_{1},a,x_{1})$ for some $a\in X$ (the existence of $a$ follows by passing to a subsequence). We conclude that $(\vec{x},\vec{y})\in \mathcal{R}_{\widehat{S}}(Y)$. Similarly $(\vec{x},\vec{y})\in \mathcal{R}_{\widehat{T}}(Y)$.
\end{proof}

Moreover, if $(X,S,T)$ is a system with commuting transformations $S$ and $T$ and $(Y,\widehat{S},\widehat{T})$ is the magic extension described in Proposition \ref{MagicExtension}, we have:
\begin{cor}
If $((x_1,x_2,x_3),(x_1,x_2,y_3))\in\Q_{\widehat{S}}(Y)$, then $((x_1,x_2,x_3),(x_1,x_2,y_3))\in \mathcal{R}_{\widehat{S}}(Y)$.
\end{cor}

The following lemma is proved implicitly in Proposition \ref{MagicExtension}. We state it here for use in the sequel:

\begin{lem} \label{R_Magic}
 Let $(X,S,T)$ be a minimal system with commuting transformations $S$ and $T$. Let $(Y,\widehat{S},\widehat{T})$ be the magic extension given by Proposition \ref{MagicExtension} and let $\vec{x}=(x_1,x_2,x_3)$, $\vec{y}=(y_1,y_2,y_3)$ be points in $Y$. For $R$ being either $S$ or $T$, if $(\vec{x},\vec{y})\in \mathcal{R}_{\widehat{R}}(Y)$ then $x_1=y_1$, $x_2=y_2$ and $(x_3,y_3)\in \mathcal{R}_R(X)$.
\end{lem}

\subsection{Partially distal systems}
We first recall the definition of proximal pairs and distality:
\begin{defn}
  Let $(X,G_{0})$ be a topological dynamical system, where $G_{0}$ is an arbitrary group action. We say that two points $x,y\in X$ are {\it proximal} if there exists a sequence $(g_i)_{i\in\mathbb{N}}$ in $G_{0}$ such that $\lim_{i\to\infty}d(g_ix,g_iy)=0$.
  
A topological dynamical system is {\it distal} if the proximal relation coincides with the diagonal relation $\Delta_X=\{(x,x)\in X\times X\colon x \in X\}$. Equivalently, $(X,G_{0})$ is distal if $x\neq y$ implies that $$\inf_{g \in G_{0}} d(gx,gy)>0.$$
\end{defn}

We introduce a definition of partial distality, which can be viewed as a generalization of distality, and is the main ingredient in the proof of Theorem \ref{R-ProductExtension}.

Let $(X,S,T)$ be a minimal system with commuting transformations $S$ and $T$. For $R$ being either $S$ or $T$, let $P_R(X)$ be the set of proximal pairs under $R$.

\begin{defn}Let $(X,S,T)$ be a minimal system with commuting transformations $S$ and $T$.  We say that $(X,S,T)$ is {\it partially distal} if $\Q_S(X)\cap P_T(X)=\Q_T(X)\cap P_S(X)=\Delta_X$.
\end{defn}
We remark that when $S=T$, partial distality coincides with distality.
If $\Q_S(X)$ is an equivalence relation on $X$, then the system $(X,S,T)$ being partially distal implies that the quotient map $X\to X/\Q_S(X)$ is a distal extension between the systems $(X,T)$ and $(X/\Q_S(X),T)$.

The following lemma allows us to lift a minimal idempotent in $E(X,G)$ to a minimal idempotent in $E(X^4,\mathcal{F}_{S,T})$. Recall that taking $R$ to be either $S$ or $T$, if $u\in E(X,R)$ is an idempotent, then $(x,ux)\in P_R(X)$ for all $x\in X$ (Theorem \ref{Enveloping2}).

\begin{lem} \label{LiftingIdempotent}
 Let $(X,S,T)$ be a minimal system with commuting transformations $S$ and $T$, and let $u\in E(X,G)$ be a minimal idempotent. Then there exists a minimal idempotent $\widehat{u}\in E(X^4,\mathcal{F}_{S,T})$ of the form $\widehat{u}=(e,u_S,u_T,u)$, where $u_S \in E(X,S)$ and $u_T \in E(X,T)$ are minimal idempotents. Moreover, if $(X,S,T)$ is partially distal, we have that $u_Su=u_Tu=u$.
\end{lem}

\begin{proof}
For $i=0,1,2,3,$ let $\pi_i$ be the projection from $X^4$ onto the $i$-th coordinate and let $\pi_i^{\ast}$ be the induced factor map in the enveloping semigroups. Hence $\pi_2^{\ast}\colon E(X^4,\mathcal{F}_{S,T})\to E(X,S)$, $\pi_3^{\ast}\colon E(X^4,\mathcal{F}_{S,T})\to E(X,T)$, and $\pi_4^{\ast}\colon E(X^4,\mathcal{F}_{S,T})\to E(X,G)$ are factor maps. By Proposition \ref{liftingIdempotent}, we can find a minimal idempotent $\widehat{u}\in E(X^4,\mathcal{F}_{S,T})$ such that $\pi_4^{\ast}(\widehat{u})=u$. Since the projection of a minimal idempotent is a minimal idempotent,  $\widehat{u}$ can be written in the form $\widehat{u}=(e,u_S,u_T,u),$ where $u_S\in E(X,S)$ and $u_T\in E(X,T)$ are minimal idempotents.

Now suppose that $(X,S,T)$ is partially distal. Let $u\in E(X,G)$ and $\widehat{u}=(e,u_S,u_T,u)\in (X^4,\mathcal{F}_{S,T})$ be minimal idempotents in the corresponding enveloping semigroups. Note that $(ux,u_Sux,u_Tux,uux)=(ux,u_Sux,u_Tux,ux)\in \Q_{S,T}(X)$ for all $x\in X$. So we have that $(ux,u_Sux)\in P_S(X)\cap \Q_T(X)$ and $(ux,u_Tux)\in P_T(X)\cap\Q_S(X)$. Thus $u_Sux=u_Tux=ux$ for all $x\in X$ since $X$ is partially distal. This finishes the proof.
\end{proof}

 \begin{cor} \label{Kminimal}
 Let $(X,S,T)$ be a partially distal system with commuting transformations $S$ and $T$. Then for every $x\in X$, the system $(\bold{K}^x_{S,T},\widehat{S}=S\times \id\times S,\widehat{T}=\id\times T\times T)$ with commuting transformations $\widehat{S}$ and $\widehat{T}$ is a minimal system. Moreover, $(\bold{K}^x_{S,T},\widehat{S},\widehat{T})$ is a magic extension of $(X,S,T)$.
\end{cor}

\begin{proof}
 Since $(X,S,T)$ is a minimal system, there exists a minimal idempotent $u\in E(X,G)$ such that $ux=x$. By Lemma \ref{LiftingIdempotent}, there exists a minimal idempotent $\widehat{u}\in E(X^4,\mathcal{F}_{S,T})$ such that $\widehat{u}(x,x,x)=(x,x,x)$, which implies that $(x,x,x)$ is a minimal point of $\bold{K}^x_{S,T}$.
 The proof that $(\bold{K}^x_{S,T},\widehat{S},\widehat{T})$ is a magic extension is similar to Proposition \ref{MagicExtension}.

\end{proof}

\begin{cor}
 Let $(X,S,T)$ be a partially distal system. Then $(X,S)$ and $(X,T)$ are pointwise almost periodic.
\end{cor}

\begin{proof}
 By Lemma \ref{LiftingIdempotent}, for any $x\in X$, we can find minimal idempotents $u_S \in E(X,S)$ and $u_T \in E(X,T)$ such that $u_Sx=u_Tx=x$. This is equivalent to being pointwise almost periodic.
\end{proof}

\subsection{Proof of Theorem \ref{R-ProductExtension}}
Before completing the proof of Theorem \ref{R-ProductExtension}, we start with some lemmas:

\begin{lem} \label{temp}
For any minimal system $(X,S,T)$ with commuting transformations $S$ and $T$, $\Q_S(X)\cap P_T(X)\subseteq\mathcal{R}_S(X)$.
\end{lem}

\begin{proof}
  Suppose $(x,y)\in\Q_S(X)\cap P_T(X)$. Since $(x,y)\in P_T(X)$, passing to a subsequence if necessary, there exists a sequence $(m_i)_{i\in \N}$ in $\Z$ such that $d(T^{m_i}x,T^{m_i}y)\to 0$. We can assume that $T^{m_i}x$ and $T^{m_i}y$ converge to $a\in X$. Since $(x,y)\in \Q_S(X)$, we have that $(x,y,x,y)\in \Q_{S,T}(X)$ and therefore $(x,y,T^{m_i}x,T^{m_i}y)\to (x,y,a,a)\in \Q_{S,T}(X)$. We conclude that $(x,y)\in \mathcal{R}_S(X)$.
\end{proof}

\begin{lem} \label{MinimalFace}
 Let $(X,S,T)$ be a minimal system with commuting transformations $S$ and $T$ such that $\mathcal{R}_S(X)=\Delta_X$. Then for every $x\in X$, $(\bold{K}^x_{S,T}, \widehat{S},\widehat{T})$ is a minimal system. Particularly, for every $x\in X$ we have that $(\mathcal{O}_S(x),S)$ and $(\mathcal{O}_T(x),T)$ are minimal systems.
\end{lem}

\begin{proof}
Since $\mathcal{R}_S(X)=\Delta_X$, by Lemma \ref{temp}, we deduce that $\Q_S(X)\cap P_T(X)=\Delta_X$. For any $x\in X$, let $u\in E(X,G)$ be a minimal idempotent with $ux=x$ and let $(e,u_S,u_T,u)\in E(X^4,\mathcal{F}_{S,T})$ be a lift given by Lemma \ref{LiftingIdempotent}. Then $(x,u_Sx,u_Tx,ux)=(x,u_Sx,u_Tx,x)\in\Q_{S,T}(X)$. Projecting to the last two coordinates, we get that $(u_Tx,x)\in\Q_{S}(X)$. On the other hand, $(u_Tx,x)\in P_{T}(X)$ as $u_T\in E(X,T)$. Since $\Q_S(X)\cap P_T(X)=\Delta_X$, we deduce that $x=u_Tx$ and thus $(x,u_Sx,u_Tx,ux)=(x,u_Sx,x,x)$. Since $\mathcal{R}_S(X)=\Delta_X$, we have that $(u_Sx,u_Tx,ux)=(x,x,x)$ and this point is minimal.

The second statement follows by projecting $\bold{K}^x_{S,T}$ onto the two first coordinates.
\end{proof}

\begin{lem} \label{Q_ST-R_S}
 Let $(X,S,T)$ be a minimal system with commuting transformations $S$ and $T$. If $\Q_{S}(X)\cap\Q_{T}(X)=\Delta_X$, then $\mathcal{R}_{S}(X)=\Delta_X$.
\end{lem}

\begin{proof}
 We remark that if $(x,a,b,x)\in \Q_{S,T}(X)$, then $(x,a)$ and $(x,b)$ belong to $\Q_{S}(X)\cap\Q_{T}(X)$. Consequently, if $(x,a,b,x)\in \Q_{S,T}(X)$, then $a=b=x$. Now let $(x,y)\in \mathcal{R}_{S}(X)$ and let $a\in X$ such that $(x,y,a,a)\in \Q_{S,T}(X)$. By minimality we can take two sequences $(n_i)_{i\in \N}$ and $(m_i)_{i\in \N}$ in $\Z$ such that $S^{n_i}T^{m_i}a\to x$. We can assume that $S^{n_i}y\to y'$ and $T^{m_i}a\to a'$, and thus $(x,S^{n_i}y,T^{m_i}a,S^{n_i}T^{m_i}a)\to (x,y',a',x)\in \Q_{S,T}(X)$. We deduce that $y'=a'=x$ and particularly $T^{m_i}a\to x$. Hence $(x,y,T^{m_i}a,T^{m_i}a)\to (x,y,x,x)\in \Q_{S,T}(X)$ and therefore $x=y$.
\end{proof}

We are now ready to prove Theorem \ref{R-ProductExtension}:

\begin{proof}[Proof of Theorem \ref{R-ProductExtension}]
\quad

$(1) \Rightarrow (2)$. Let $\pi\colon Y\times W\to X$ be a factor map between the minimal systems $(Y\times W ,\sigma\times \id, \id \times \tau)$ and $(X,S,T)$. Let $(x_0,x_1,x_2,x_3)$ and $(x_0,x_1,x_2,x_3')\in \Q_{S,T}(X)$. It suffices to show that $x_3=x_3'$. Since $\pi^4(\Q_{\sigma\times \id ,\id \times \tau}(Y\times W))=\Q_{S,T}(X)$, there exist $\left ( (y_0,w_0),(y_1,w_0),(y_0,w_1),(y_1,w_1) \right)$ and $ ((y_0',w_0'),(y_1',w_0'),(y_0',w_1'),(y_1',w_1') )$ in $\Q_{\sigma\times \id ,\id \times \tau}(Y\times W) $ such that $\pi(y_0,w_0)=x_0=\pi(y_0',w_0')$, $\pi(y_1,w_0)=x_1=\pi(y_1',w_0')$,  $\pi(y_0,w_1)=x_2=\pi(y_0',w_1')$, $\pi(y_1,w_1)=x_3$ and $\pi(y_1',w_1')=x_3'$.

 Let $(n_i)_{i\in \N}$ and $(m_i)_{i\in \N}$ be sequences in $\Z$ such that $\sigma^{n_i}y_0\to y_1$ and $\tau^{m_i}w_0\to w_1$. We can assume that $\sigma^{n_i}y_0'\to y_1''$ and $\tau^{m_i}w_0'\to w_1''$ so that $((y_0',w_0'),(y_1'',w_0'),(y_0',w_1''),(y_1'',w_1'') ) \in \Q_{\sigma\times \id ,\id \times \tau}(Y\times W) $. Since $\pi(y_0,w_0)=\pi(y_0',w_0')$, we have that $$\pi^4((y_0',w_0'),(y_1'',w_0'),(y_0',w_1''),(y_1'',w_1''))=(x_0,x_1,x_2,x_3).$$

 Particularly, $\pi(y_1',w_0')=\pi(y_1'',w_0')$ and $\pi(y_0',w_1')=\pi(y_0',w_1'')$. By minimality of $(Y,\sigma)$ and $(W,\tau)$, we deduce that $\pi(y_1',w)=\pi(y_1'',w)$ and $\pi(y,w_1')=\pi(y,w_1'')$ for every $y\in Y$ and for every $w\in W$. Hence $x_3=\pi(y_1'',w_1'')=\pi(y_1'',w_1')=\pi(y_1',w_1')=x_3'$.

$(2)\Rightarrow (3)$. Let $(x,y)\in \mathcal{R}_{S}(X)$ and let $a\in X$ such that $(x,y,a,a)\in \Q_{S,T}(X)$. We remark that this implies that $(x,a)\in \Q_T(X)$ and then $(x,x,a,a)\in \Q_{S,T}(X)$. Since $(x,x,a,a)$ and $(x,y,a,a)$ belong to $\Q_{S,T}(X)$, we have that $x=y$.

$(3)\Rightarrow (1).$ By Lemma \ref{MinimalFace}, for every $x_0 \in X$, we can build a minimal magic system $(\bold{K}^{x_0}_{S,T},\widehat{S},\widehat{T})$ which is an extension of $(X,S,T)$ whose factor map is the projection onto the last coordinate. We remark that if $\vec{x}=(x_1,x_2,x_3)$ and $\vec{y}=(y_1,y_2,y_3)$ are such that $(\vec{x},\vec{y})\in \mathcal{R}_{\widehat{S}}(\bold{K}^{x_0}_{S,T})$, then by Lemma \ref{R_Magic}, $x_1=y_1$, $x_2=y_2$ and $(x_3,y_3)\in \mathcal{R}_S(X)$. Hence, if $\mathcal{R}_S(X)$ coincides with the diagonal, so does $\mathcal{R}_{\widehat{S}}(\bold{K}^{x_0}_{S,T})$.

Let $\phi\colon\bold{K}^{x_0}_{S,T}\to \mathcal{O}_{S}(x_0)\times\mathcal{O}_{T}(x_0)$ be the projection onto the first two coordinates. Then $\phi$ is a factor map between the minimal systems $(\bold{K}^{x_0}_{S,T},\widehat{S},\widehat{T})$ and $(\mathcal{O}_{S}(x_0)\times\mathcal{O}_{T}(x_0), S\times \id, \id \times T)$ with commuting transformations. We remark that the latter is a product system.

 We claim that the triviality of the relation $\mathcal{R}_S(X)$ implies that $\phi$ is actually an isomorphism. It suffices to show that $(a,b,c), (a,b,d)\in \bold{K}^{x_0}_{S,T}$ implies that $c=d$. By minimality, we can find a sequence $(n_i)_{i\in\mathbb{N}}$ in $\Z$ such that $S^{n_i}a\to x_0$. Since $\mathcal{R}_S(X)=\Delta_X$, we have that $\lim S^{n_i}c=b=\lim S^{n_i} d$. So $\lim \widehat{S}^{n_i}(a,b,c)=\lim \widehat{S}^{n_i}(a,b,d)$ and hence $\left ( (a,b,c), (a,b,d) \right ) \in P_{\widehat{S}}(\bold{K}^{x_0}_{S,T})$. Since $\mathcal{R}_{\widehat{S}}(\bold{K}^{x_0}_{S,T})$ is the diagonal, by Lemma \ref{MinimalFace} applied to the system $(\bold{K}^{x_0}_{S,T},\widehat{S},\widehat{T})$ we have that every point in $\bold{K}^{x_0}_{S,T}$ has a minimal $\widehat{S}$-orbit. This implies that $(a,b,c)$ and $(a,b,d)$ are in the same $\widehat{S}$-minimal orbit closure and hence they belong to $\Q_{\widehat{S}}(\bold{K}^{x_0}_{S,T})$. By Proposition \ref{MagicExtension}, since they have the same first two coordinates, we deduce that $( (a,b,c), (a,b,d)) \in \mathcal{R}_{\widehat{S}}(\bold{K}^{x_0}_{S,T})$, which is trivial. We conclude that $(\bold{K}^{x_0}_{S,T},\widehat{S},\widehat{T})$ is a product system and thus $(X,S,T)$ has a product extension.

 $(2)\Rightarrow (4)$ is similar to $(2)\Rightarrow (3)$; $(4)\Rightarrow (1)$ is similar to $(3)\Rightarrow (1)$; $(3)\Rightarrow (5)$ is obvious.

 $(5) \Rightarrow (1)$. By Proposition \ref{MagicExtension}, we have a magic extension $(Y,\widehat{S},\widehat{T})$ of $(X,S,T)$ with $Y\subseteq \bold{K}^{x_0}_{S,T}$ for some $x_0 \in X$. The magic extension satisfies $\Q_{\widehat{S}}(Y)\cap\Q_{\widehat{T}}(Y)=\mathcal{R}_{\widehat{S}}(Y)\cap \mathcal{R}_{\widehat{T}}(Y)$. Since $\mathcal{R}_S(X)\cap\mathcal{R}_T(X)$ is the diagonal, by Lemma \ref{R_Magic}, we have that $\mathcal{R}_{\widehat{S}}(Y)\cap \mathcal{R}_{\widehat{T}}(Y)=\Q_{\widehat{S}}(Y)\cap\Q_{\widehat{T}}(Y)$ is also the diagonal. By Lemma \ref{Q_ST-R_S}, we have that $\mathcal{R}_{\widehat{S}}(Y)$ coincides with the diagonal relation. Therefore, $(Y,\widehat{S},\widehat{T})$ satisfies property (3) and we have proved above that this implies that $(Y,\widehat{S},\widehat{T})$ (and consequently $(X,S,T)$) has a product extension. This finishes the proof.
\end{proof}

We remark that if $(X,S,T)$ has a product extension, then Theorem \ref{R-ProductExtension} gives us an explicit (or algorithmic) way to build such an extension. In fact, we have:

\begin{prop} \label{FunctionCoordinate}
 Let $(X,S,T)$ be a minimal system with commuting transformations $S$ and $T$. The following are equivalent:

   \begin{enumerate}
    \item $(X,S,T)$ has a product extension;

    \item There exists $x\in X$ such that the last coordinate of $\bold{K}^{x}_{S,T}$ is a function of the first two coordinates. In this case, $(\bold{K}^{x}_{S,T},\widehat{S},\widehat{T})$ is a product system;

    \item For any $x\in X$, the last coordinate of $\bold{K}^{x}_{S,T}$ is a function of the first two coordinates.  In this case, $(\bold{K}^{x}_{S,T},\widehat{S},\widehat{T})$ is a product system.

   \end{enumerate}

\end{prop}

\begin{proof}

$(1)\Rightarrow (3)$. By Theorem \ref{R-ProductExtension}, when $(X,S,T)$ has a product extension, then the last coordinate of $\Q_{S,T}(X)$ is a function of the first three ones, which implies (3).

 $(3)\Rightarrow (2)$. Is obvious.

 $(2)\Rightarrow (1)$. Let $Y\subseteq \bold{K}^{x}_{S,T}$ be a minimal subsystem and let $(x_1,x_2,x_3)\in Y$. We remark that $(Y,\widehat{S},\widehat{T})$ is an extension of $(X,S,T)$ and that the last coordinate of $Y$ is a function of the first two coordinates. Hence, the factor map $(x_1',x_2',x_3')\to (x_1',x_2')$ is an isomorphism between $(Y,\widehat{S},\widehat{T})$ and $(\mathcal{O}_S(x_1)\times \mathcal{O}_T(x_2), S\times \id,\id \times T)$, which is a product system.

\end{proof}

We can also give a criterion to determine when a minimal system $(X,S,T)$ with commuting transformations $S$ and $T$ is actually a product system:

\begin{prop}
 Let $(X,S,T)$ be a minimal system with commuting transformations $S$ and $T$. Then $(X,S,T)$ is a product system if and only if $\Q_{S}(X)\cap \Q_{T}(X)=\Delta_X$.
\end{prop}

\begin{proof}
 Suppose that $(X,S,T)=(Y\times W,\sigma\times \id,\id\times \tau)$ is a product system and $(y_{1},w_{1})$, $(y_{2},w_{2})\in\Q_{\sigma\times \id}(Y\times W)\cap \Q_{\id\times \tau}(Y\times W)$. Then
 $((y_{1},w_{1}),(y_{2},w_{2}))\in\Q_{\id\times \tau}(Y\times W)$ implies that $y_{1}=y_{2}$, and $((y_{1},w_{1})$, $(y_{2},w_{2}))\in\Q_{\sigma\times \id}(Y\times W)$ implies that $w_{1}=w_{2}$. Therefore, $\Q_{S}(Y\times W)\cap \Q_{T}(Y\times W)=\Delta_{Y\times W}$.

 Conversely, suppose that $\Q_{S}(X)\cap \Q_{T}(X)=\Delta_X$. By Lemma \ref{Q_ST-R_S}, Theorem \ref{R-ProductExtension} and Proposition \ref{FunctionCoordinate}, we have that for any $x_0\in X$, $(\bold{K}^{x_0}_{S,T},\widehat{S},\widehat{T})$ is a product extension of $(X,S,T)$. We claim that these systems are actually isomorphic. Recall that the factor map $\pi\colon \bold{K}^{x_0}_{S,T}\to X$ is the projection onto the last coordinate. It suffices to show that $(x_1,x_2)=(x_1',x_2')$ for all $(x_1,x_2,x),(x_1',x_2',x)\in \bold{K}^{x_0}_{S,T}$. Let $(n_i)_{i\in \N}$ and $(m_i)_{i\in \N}$ be sequences in $\Z$ such that $S^{n_i}T^{m_i}x\to x_0$. We can assume that $S^{n_i}x_1\to a_1$, $S^{n_i}x_1'\to a_1'$, $T^{m_i}x_2\to b_2$ and $T^{m_i}x_2'\to b_2'$. Therefore, $(x_0,a_1,b_1,x_0)$ and $(x_0,a_1',b_1',x_0)$ belong to $\Q_{S,T}(X)$. Since $\Q_{S}(X)\cap \Q_{T}(X)=\Delta_X$, we have that $a_1=b_1=a_1'=b_1'=x_0$. We can assume that $S^{n_i}x\to x'$ and thus $(x_0,S^{n_i}x_1,x_2,S^{n_i}x)\to (x_0,x_0,x_2,x')$,  $(x_0,S^{n_i}x_1',x_2',S^{n_i}x)\to (x_0,x_0,x_2',x')$. Moreover, these points belong to $\Q_{S,T}(X)$. Since $\mathcal{R}_{S}(X)$ is the diagonal, we conclude that $x_2=x'=x_2'$. Similarly, $x_1=x_1'$ and the proof is finished.

\end{proof}

\subsection{Equicontinuity and product extensions}
Let $(X,S,T)$ be a system with commuting transformations $S$ and $T$. Let suppose that $(X,S,T)$ has a product extension. In this section we show that one can always find a product extension where the factor map satisfies some kind of equicontinuity conditions.

We recall the definition of equicontinuity:
\begin{defn}
    Let $(X,G_{0})$ be a topological dynamical system, where $G_{0}$ is an arbitrary group action. We say that $(X,G_{0})$ is {\it equicontinuous} if for any $\e>0$, there exists $\d>0$ such that if $d(x,y)<\d$ for $x,y\in X$, then $d(gx,gy)<\e$ for all $g\in G_{0}$.  Let $\pi\colon Y\to X$ be a factor map between the topological dynamical systems $(Y,G_{0})$ and $(X,G_{0})$ . We say that $Y$ is an {\it equicontinuous extension} of $X$ if for any $\e>0$, there exists $\d>0$ such that if $d(x,y)<\d$ and $\pi(x)=\pi(y)$ then $d(gx,gy)<\e$ for all $g\in G_{0}$.
\end{defn}

The following proposition provides the connection between equicontinuity and the property of being a factor of a product system:

\begin{prop}
Let $(X,S,T)$ be a minimal system with commuting transformations $S$ and $T$. If either $S$ or $T$ is equicontinuous, then $(X,S,T)$ has a product extension.
\end{prop}

\begin{proof}
 Suppose that $T$ is equicontinuous. For any $\e>0$, let $0<\delta<\e$ be such that if two points are $\delta$-close to each other, then they stay $\e$-close under the orbit of $T$. Suppose $(x,y)\in \mathcal{R}_S(X)$. Pick $x',a\in X$ and $n,m\in \Z$ such that $d(x,x')<\delta$, $d(S^nx',y)<\delta$, $d(T^mx',a)<\delta$, $d(S^n T^m x',a)<\delta$. By equicontinuity of $T$, we have that $d(T^{-m}S^nT^mx',T^{-m}a)<\e$, $d(T^{-m}T^mx',T^{-m}a)<\e$. Therefore $d(x,y)<4\e$. Hence, $\mathcal{R}_S(X)$ coincides with the diagonal and $(X,S,T)$ has a product extension.
\end{proof}

Specially, when $S=T$ we have:

\begin{cor} \label{EquicontinuousAndProduct}
 Let $(X,T)$ be a minimal system. Then $(X,T)$ is equicontinous if and only if $(X,T,T)$ has a product extension.
\end{cor}

Under the assumption that $\Q_{T}(X)$ is an equivalence relation, we have a better criterion:

\begin{prop} \label{Eqextension}
 Let $(X,S,T)$ be a minimal system with commuting transformations $S$ and $T$. Suppose that $\Q_{T}(X)$ is an equivalence relation. Then the system $(X,S)$ is an equicontinuous extension of $(X/\Q_{T}(X),S)$ if and only if $(X,S,T)$ has a product extension.
\end{prop}

\begin{proof}
  Suppose that $(X,S,T)$ has no product extensions. By Theorem \ref{R-ProductExtension}, we can pick $x,y\in X, x\neq y$ such that $(x,y)\in\mathcal{R}_{T}(X)$. Denote $\e=d(x,y)/2$. For any $0<\d<\e/4$, there exist $z\in X, n,m\in \Z$ such that $d(z,x)$, $d(T^{m}z,y)$, $d(S^{n}z,S^{n}T^{m}z)<\d$. Let $x'=S^{n}z,y'=S^{n}T^{m}z$. Then $(x',y')\in\Q_{T}(X)$, $d(x',y')<\delta$ and $d(S^{-n}x',S^{-n}y')=d(z,T^{m}z)>\e-2\delta>\e/2$. So $(X,S)$
  is not an equicontinuous extension of $(X/\Q_T(X),S)$.

  On the other hand, if $(X,S)$ is not an equicontinuous extension of $(X/\Q_T(X),S)$, then there exists $\e>0$ and there exist sequences $(x_{i})_{i\in \N}, (y_{i})_{i\in \N}$ in $X$ and a sequence $(n_{i})_{i\in \N}$ in $\Z$ with $d(x_{i},y_{i})<1/i$, $(x_{i},y_{i})\in\Q_{T}(X)$, and $d(S^{n_{i}}x_{i},S^{n_{i}}y_{i})\geq\e$. By passing to a subsequence, we may assume $(S^{n_{i}}x_{i})_{i\in\N}$, $(S^{n_{i}}y_{i})_{i\in\N}$, $(x_{i})_{i\in\N}$ and $(y_{i})_{i\in\N}$ converges to $x_{0},y_{0}, w$ and $w$ respectively. Then $x_{0}\neq y_{0}$. For any $\d>0$, pick $i\in\N$ such that
  $d(S^{n_{i}}x_{i},x_{0})$, $d(S^{n_{i}}y_{i},y_{0})$, $d(x_{i},w)$, $d(y_{i},w)<\d$. Since $(x_{i},y_{i})\in\Q_{T}(X)$, we can pick $z\in X, m\in\Z$ such that $d(z,x_{i})$, $d(T^{m}z,y_{i})$, $d(S^{n_{i}}z$, $S^{n_{i}}x_{i})$, $d(S^{n_{i}}T^{m}z, S^{n_{i}}y_{i})<\d$. So the distance between the corresponding coordinates of $(S^{n_{i}}z,z,S^{n_{i}}T^{m}z,T^{m}z)$ and $(x_{0},w,y_{0},w)$ are all less than $C\d$ for some uniform constant $C$. So $(x_{0},y_{0})\in\mathcal{R}_{T}(X)$, and $(X,S,T)$ has not a product extension.
\end{proof}

In the following we relativize the notion of being a product system to factor maps.

\begin{defn}
Let $\pi\colon Y\to X$ be a factor map between the systems of commuting transformations $(Y,S,T)$ and $(X,S,T)$. We say that $\pi$ is {\it $S$-equicontinuous with respect to $T$} if for any $\e>0$ there exists $\delta>0$ such that if $y,y'\in Y$ satisfy
$(y,y')\in \Q_T(Y)$, $d(y,y')<\delta$ and $\pi(y)=\pi(y')$, then $d(S^ny,S^ny')<\e$ for all $n\in \Z$.

\end{defn}

\begin{lem}
 Let $(X,S,T)$ be a minimal system with commuting transformations $S$ and $T$, and let $\pi$ be the projection to the trivial system. Then $\pi$ is $S$-equicontinous with respect to $T$  if and only if $(X,S,T)$ has a product extension.
\end{lem}

\begin{proof}
If $\pi$ is not $S$-equicontinuous with respect to $T$, there exists $\e>0$ such that for any $\delta=\frac{1}{i}>0$ one can find $(x_i,x_i')\in \Q_T(X)$ with $d(x_i,x_i')<\delta$ and $n_i\in \Z$ with $d(S^{n_i}x_i,S^{n_i}x_i')\geq \e$. For a subsequence, $(x_i,S^{n_i}x_i,x_i',S^{n_i}x_i')\in \Q_{S,T}(X)$ converges to a point of the form $(a,x,a,x')\in \Q_{S,T}(X)$ with $x\neq x'$. We remark that this is equivalent to $(x,a,x',a)\in \Q_{S,T}(X)$ and hence $(x,x')\in \mathcal{R}_S(X)$. By Theorem \ref{R-ProductExtension} $(X,S,T)$ has no product extension.

Conversely, if $(X,S,T)$ has no product extension, by Theorem \ref{R-ProductExtension} we can find $x\neq x'$ with $(x,x')\in \mathcal{R}_S(X)$. Let $0<\e<d(x,x')$ and let $0<\delta<\e/4$. We can find $x''\in X$ and $n,m\in \Z$ such that $d(x'',x)<\delta$, $d(S^nx'',x')<\delta$ and $d(T^mx'',S^nT^mx'')<\delta$. Writing $w=T^mx''$, $w'=S^nT^mx''$, we have that $(w,w')\in \Q_S(X)$, $d(w,w')<\delta$ and $d(T^{-m}w,T^{-m}w')>\e/2$. Hence $\pi$ is not $S$-equicontinuous with respect to $T$.
\end{proof}

A connection between a magic system and a system which is $S$-equicontinuous with respect to $T$ is:

\begin{prop} \label{KExtension}
 For every minimal system with commuting transformations $(X,S,T)$, the magic extension constructed in Theorem \ref{MagicExtension} is $S$-equicontinuous with respect to $T$.
\end{prop}

\begin{proof}Let $(X,S,T)$ be a minimal system with commuting transformations $S$ and $T$.
Recall that the magic extension $Y$ of $X$ is the orbit closure of a minimal point $(z_1,z_2,x_0)$ under $\widehat{S}$ and $\widehat{T}$, and the factor map $\pi\colon Y\to X$  is the projection onto the last coordinate. Let $\vec{x}=(x_1,x_2,x_3),\vec{y}=(y_1,y_2,y_3)\in Y$ be such that $\pi(\vec{x})=\pi(\vec{y})$ and $(\vec{x},\vec{y})\in \Q_{\widehat{T}}(Y)$. Then we have that $x_1=y_1$ and $x_3=y_3$. Since $\widehat{S}^n\vec{x}=(S^nx_1,x_2,S^nx_3)$ and $\widehat{S}^n\vec{y}=(S^nx_1,y_2,S^nx_3)$, we conclude that $\widehat{S}$ preserves the distance between $\vec{x}$ and $\vec{y}$.
\end{proof}

A direct corollary of this proposition is:
\begin{cor} \label{KExtension}
 Let $(X,S,T)$ be a minimal system with commuting transformations $S$ and $T$. If $(X,S,T)$ has a product extension, then it has a product extension which is $S$-equicontinuous with respect to $T$.
\end{cor}

\begin{proof}
If $(X,S,T)$ has a product extension, by Theorem \ref{R-ProductExtension}, we can build a magic extension which is actually a product system. This magic extension is $S$-equicontinuous with respect to $T$.                                                                                                                                                                                                                                                                                                    \end{proof}

\subsection{Changing the generators}
Let $(X,S,T)$ be a system with commuting transformations $S$ and $T$. We remark that $\Q_{S,T}(X)$ depends strongly on the choice of the generators $S$ and $T$. For instance, let $(X,S)$ be a minimal system and consider the minimal systems $(X,S,S)$ and $(X,S,\id)$ with commuting transformations. We have that $(X,S,\id)$ has a product extension, but $(X,S,S)$ does not (unless $(X,S)$ is equicontinous). However, there are cases where we can deduce some properties by changing the generators. Let $(X,S,T)$ be a minimal system with commuting transformations $S$ and $T$. Denote $S'=T^{-1}S$, $T'=T$. We have that $(X,S',T')$ is a minimal system with commuting transformations $S'$ and $T'$. Suppose now that $(X,S',T')$ has a product extension. By Proposition \ref{FunctionCoordinate}, for any $x\in X$ we have that $(K^{x}_{S',T'},\widehat{S'},\widehat{T'})$ is an extension of $(X,S',T')$ and it is isomorphic to a product system. We remark that $(K^{x}_{S',T'},\widehat{T'}\widehat{S'},\widehat{T'})$ is an extension of $(X,S,T)$ and it is isomorphic to $(Y\times W,S\times T, T \times T)$, where $Y=\mathcal{O}_{S'}(x)$ and $W=\mathcal{O}_{T'}(x)$. It follows that $(X,S,T)$ has an extension which is the Cartesian product of two systems with commuting transformations with different natures: one of the form $(Y,S,\id)$ where one of the transformations is the identity, and the other of the form $(W,T,T)$ where the two transformations are the same.

\bigskip
\subsection{Computing the group of automorphisms by using the $\mathcal{R}_{S,T}(X)$ relation}
The following lemma is used in the next section to study the automorphism group of the Robinson tiling system, but we state it here due to its generality. It reveals that studying cube structures can help to understand the group of automorphisms of a dynamical system. We recall that an {\it automorphism} of a dynamical system $(X,G_0)$ is a homeomorphism $\phi\colon X \to X$ such that $\phi g =g \phi $ for every $g\in G_0$.

\begin{lem} \label{AutoFiber}
 Let $(X,S,T)$ be a minimal system with commuting transformations $S$ and $T$, and let $\phi$ be an automorphism of $(X,S,T)$. Then $\phi\times \phi \times \phi\times \phi (\Q_{S,T}(X))=\Q_{S,T}(X)$. Particularly, if $(x,y)\in \mathcal{R}_S(X)$ (or $\mathcal{R}_T(X)$ or $\mathcal{R}_{S,T}(X)$), then $(\phi(x),\phi(y)) \in \mathcal{R}_S(X)$ (or $\mathcal{R}_T(X)$ or $\mathcal{R}_{S,T}(X)$).
\end{lem}

\begin{proof}
 Let ${\bf x}\in \Q_{S,T}(X)$ and let $x\in X$. There exist sequences $(g_i)_{i\in\mathbb{N}}$ in $G$ and $(n_i)_{i\in\mathbb{N}},(m_i)_{i\in\mathbb{N}}$ in $\Z$ such that $(g_ix,g_iS^{n_i}x,g_iT^{m_i}x,g_iS^{n_i}T^{m_i}x)\to {\bf x}$. Since $(\phi(x),\phi(x),\phi(x),\phi(x))\in \Q_{S,T}(X)$ we have that
 \begin{align*}
 & (g_i\phi(x),g_iS^{n_i}\phi(x),g_iT^{m_i}\phi(x),g_iS^{n_i}T^{m_i}\phi(x)) \in \Q_{S,T}(X) \\
 = &(\phi(g_ix),\phi(g_iS^{n_i}x),\phi(g_iT^{m_i}x),\phi(g_iS^{n_i}T^{m_i}x)) \in \Q_{S,T}(X) \\
 \to & (\phi\times \phi \times \phi \times \phi )( {\bf x}) \in \Q_{S,T}(X).
 \end{align*}
Hence $\phi\times \phi\times \phi\times \phi(\Q_{S,T}(X))=\Q_{S,T}(X)$.

If $(x,y)\in \mathcal{R}_{S}(X)$, then there exists $a\in X$ with $(x,y,a,a)\in \Q_{S,T}(X)$ and thus $(\phi(x),\phi(y),\phi(a),\phi(a))\in \Q_{S,T}(X)$. This means that $(\phi(x),\phi(y))\in \mathcal{R}_S(X)$. The proof for the cases $\mathcal{R}_S(X)$ and $\mathcal{R}_{S,T}(X)$ are similar.
\end{proof}

\section{Examples} In this section, we compute the $\mathcal{R}_{S,T}(X)$ relation in some minimal symbolic systems $(X,S,T)$. We have chosen some representative minimal symbolic systems $(X,S,T)$ and we find that computing cube structures results useful to study some associated objects like the group of automorphisms. We start by recalling some general definitions.

Let $\mathcal{A}$ be a finite alphabet. The {\it shift transformation} $\sigma\colon \mathcal{A}^{\Z}\to \mathcal{A}^{\Z}$ is the map $(x_i)_{i\in \Z}\mapsto (x_{i+1})_{i\in \Z}$. A {\it one dimensional subshift} is a closed subset $X\subseteq \mathcal{A}^{\Z}$ invariant under the shift transformation. When there is more than one space involved, we  let $\sigma_X$ denote the shift transformation on the space $X$.

In the two dimensional setting, we define the {\it shift transformation} $\sigma_{(1,0)}\colon \mathcal{A}^{\Z^2}\to \mathcal{A}^{\Z^2}$, $(x_{i,j})_{i,j\in \Z}\mapsto (x_{i+1,j})_{i,j\in \Z}$ and $\sigma_{(0,1)}\colon\mathcal{A}^{\Z^2}\to \mathcal{A}^{\Z^2}$, $(x_{i,j})_{i,j\in \Z}\mapsto (x_{i,j+1})_{i,j\in \Z}$. Hence $\sigma_{(1,0)}$ and $\sigma_{(0,1)}$ are the translations in the canonical directions. A {\it two dimensional subshift} is a closed subset $X\subseteq \mathcal{A}^{\Z^2}$ invariant under the shift transformations. We remark that $\sigma_{(1,0)}$ and $\sigma_{(0,1)}$ are a pair of commuting transformations and therefore if $X\subseteq \mathcal{A}^{\Z^2}$ is a subshift, $(X,\sigma_{(1,0)},\sigma_{(0,1)})$ is a system with commuting transformations $\sigma_{(1,0)}$ and $\sigma_{(0,1)}$.

Let $X\subseteq \mathcal{A}^{\Z^2}$ be a subshift and let $x\in X$. If $B$ is a subset of $\Z^2$, we let $x|_B\in \mathcal{A}^{B}$ denote the restriction of $x$ to $B$ and for $\vec{n}\in \Z^2$, we let $B+\vec{n}$ denote the set $\{\vec{b}+\vec{n}\colon\vec{b}\in B\}$. When $X$ is a subshift (one or two dimensional), we let $\mathcal{A}_X$ denote its alphabet.

In the following we compute the relation $\mathcal{R}_{\sigma_{(1,0)},\sigma_{(0,1)}}(X)$ in two particular two dimensional subshifts: The Morse Tiling and the minimal Robinson tiling. For the Morse tiling and tiling substitutions, see \cite{Rad} for more background, and for the Robinson tiling we refer to \cite{GJS}, \cite{Rad}.

\subsection{The Morse tiling}

Consider the Morse tiling system given by the substitution rule:
\begin{center}
\begin{tikzpicture}

\matrix[column sep=0.0cm,row sep=0.0cm,cells={scale=0.5},ampersand replacement=\&]{
 {}\& \& \& \& \Blanco \& \Negro \& \TT  \& \TT \& \TT \& \& \TT \& \& \Negro \& \Blanco  \\
 \Blanco \& \TT \& \draw [->] (0,0) -- (1,0);  \& \TT \& \Negro \& \Blanco\& \TT \& \TT \& \TT  \Negro \& TT \& \draw [->] (0,0) -- (1,0);  \& \TT \& \Blanco \& \Negro\\
 }
;
\end{tikzpicture}
\end{center}

One can iterate this substitution in a natural way:
  \begin{figure}[H]
\begin{center}

 \begin{tikzpicture}

\matrix[column sep=0.0cm,row sep=0.0cm,cells={scale=0.5},ampersand replacement=\&]{
 {}\& \& \& \& \& \& \& \& \&  \& \&  \&  \& \& \& \& \Blanco \& \Negro \& \Negro \& \Blanco \& \Negro \& \Blanco \& \Blanco \&  \Negro \&  \&  \\
 \& \& \& \&\& \& \&\&  \&  \&  \&  \&  \& \& \& \&  \Negro \&  \Blanco \& \Blanco \&  \Negro \&\Blanco \& \Negro \&  \Negro \& \Blanco \&  \&  \\
 \& \& \& \&\&  \&  \&  \& \TT \&  \&  \&  \&   \& \& \& \&  \Negro \&  \Blanco \& \Blanco \&  \Negro \&\Blanco \& \Negro \&  \Negro \& \Blanco \&  \&  \\
\& \& \& \& \& \&  \& \&  \&    \&  \&  \&  \& \& \& \& \Blanco \& \Negro \& \Negro \& \Blanco \& \Negro \& \Blanco \& \Blanco \&  \Negro \&  \& \\
 {}\& \& \& \& \& \& \& \& \&  \Blanco \& \Negro \& \Negro \& \Blanco \& \& \& \& \Negro \&  \Blanco \& \Blanco \&  \Negro \&\Blanco \& \Negro \&  \Negro \& \Blanco \&  \&  \\
 \& \& \& \&\& \& \& \&  \& \Negro \& \Blanco \& \Blanco \& \Negro \& \& \& \& \Blanco \& \Negro \& \Negro \& \Blanco \& \Negro \& \Blanco \& \Blanco \&  \Negro \&  \& \\
\&  \& \& \& \Blanco \& \Negro \& \&  \& \&  \Negro \& \Blanco \&  \Blanco \& \Negro  \& \TT \&  \& \TT \& \Blanco \& \Negro \& \Negro \& \Blanco \& \Negro \& \Blanco \& \Blanco \&  \Negro \&  \& \\
  \Blanco \& \TT \& \draw [->] (0,0) -- (1,0); \& \TT \& \Negro \& \Blanco \& \TT \& \draw [->] (0,0) -- (1,0); \& \TT  \&   \Blanco \& \Negro \& \Negro \& \Blanco \& \& \draw [->] (0,0) -- (1,0); \& \& \Negro \&  \Blanco \& \Blanco \&  \Negro \&\Blanco \& \Negro \&  \Negro \& \Blanco \&  \&  \\
 }
;
\end{tikzpicture}

\end{center}
\caption{first, second and third iteration of the substitution}
\end{figure}
We identify 0 with the white square and 1 with the black one. Let $B_n=([-2^{n-1},2^{n-1}-1]\cap \Z)\times([-2^{n-1},2^{n-1}-1]\cap \Z)$ be the square of size $2^n$ centered at the origin. Let $(x_n)_{n\in \N}$ be a sequence in $\{0,1\}^{\Z^2}$ such that the restriction of $x_n$ to $B_n$ coincides with the $n$th-iteration of the substitution. Taking a subsequence we have that $(x_n)_{n\in \N}$ converges to a point $x^{\ast}\in \{0,1\}^{\Z^2}$. Let $X_{M}\subseteq \{0,1\}^{\Z^2}$ be the orbit closure of $x$ under the shift actions. We point out that $X_{M}$ does not depend on the particular choice of $x$ (we refer to Chapter 1 of \cite{Rad} for a general reference about substitution tiling systems). Moreover, the {\it Morse system} $(X_{M},\sigma_{(1,0)},\sigma_{(0,1)})$ is a minimal system with commuting transformations $\sigma_{(1,0)}$ and $\sigma_{(0,1)}$.

\begin{prop}
   For the Morse system, $\mathcal{R}_{\sigma_{(1,0)}}(X_{M})=\mathcal{R}_{\sigma_{(0,1)}}(X_{M})=\Delta_{X_{M}}$. Consequently, the Morse system has a product extension.
\end{prop}
\begin{proof}
Note that for $x=(x_{i,j})_{i,j\in\Z}\in X_{M}$, we have that $x_{i,j}+x_{i+1,j}=x_{i,j'}+x_{i+1,j'}\mod 2$ and $x_{i,j}+x_{i,j+1}=x_{i',j}+x_{i,j+1}\mod 2$ for every $i,j,i',j'\in \Z$. From this, we deduce that if $x_{0,0}=0$ then $x_{i,j}=x_{i,0}+x_{0,j}$ for every $i,j\in \Z$. From now on, we assume that $x^{\ast}_{0,0}=0$.

For $N\in \N$, let $B_N$ denote the square $([-N,N]\cap \Z)\times([-N,N]\cap \Z)$. Suppose $(y,z)\in\mathcal{R}_{\sigma_{(1,0)}}(X_{M})$ and let $w\in X_{M}$ be such that $(y,z,w,w)\in \Q_{\sigma_{(1,0)},\sigma_{(0,1)}}(X_{M})$. We deduce that there exist $n,m,p,q \in\mathbb{Z}$ such that
  \begin{equation}\nonumber
    \begin{split}
     & \sigma_{(1,0)}^{p}\sigma_{(0,1)}^{q}x^{\ast}|_{B_N}= y|_{B_N};
     \\&  \sigma_{(1,0)}^{p+n}\sigma_{(0,1)}^{q}x^{\ast}|_{B_N}= z|_{B_N};
     \\& \sigma_{(1,0)}^{p}\sigma_{(0,1)}^{q+m}x^{\ast}|_{B_N}=\sigma_{(1,0)}^{p+n}\sigma_{(0,1)}^{q+m}x^{\ast}|_{B_N}=w|_{B_N}.
    \end{split}
  \end{equation}
  Since $\sigma_{(1,0)}^{p}\sigma_{(0,1)}^{q+m}x^{\ast}|_{B_N}=\sigma_{(1,0)}^{p+n}\sigma_{(0,1)}^{q+m}x^{\ast}|_{B_N}$, we deduce that $x^{\ast}_{p+c,0}=x^{\ast}_{p+n+c,0}$ for all $c\leq N$. This in turn implies
  that $y|_{B_N}=\sigma_{(1,0)}^{p}\sigma_{(0,1)}^{q}x^{\ast}|_{B_N}=\sigma_{(1,0)}^{p+n}\sigma_{(0,1)}^{q}x^{\ast}|_{B_N}=z|_{B_N}$. Since $N$ is arbitrary we deduce that $y=z$. Therefore $\mathcal{R}_{\sigma_{(1,0)}}(X_{M})=\Delta_{X_{M}}$ and thus $(X_{M},\sigma_{(1,0)},\sigma_{(0,1)})$ has a product extension.
\end{proof}

\begin{rem}
 In fact, let $(Y,\sigma)$ be the one dimensional Thue-Morse system. This is the subshift generated by the one dimensional substitution $0\mapsto 01$, $1\mapsto 10$ (see \cite{Queff}). Then we can define $\pi\colon Y\times Y\to X_{M}$ by $\pi(x,x')_{n,m}=x_n+x'_m$ and it turns out that this is a product extension of the two dimensional Morse system. Moreover, we have that $(\bold{K}^{x^{\ast}}_{S,T},\widehat{S},\widehat{T})$ is  isomorphic to $(Y\times Y,T\times \id,\id\times T)$, where the isomorphism $\phi\colon\bold{K}^{x^{\ast}}_{S,T}\to Y\times Y$ is given by $\phi(a,b,c)=(a|_A,b|_B)$, where $A=\{(n,0)\colon n\in \Z\}$ and $B=\{(0,n)\colon n\in \Z\}$. We show in the next subsection that this is a general procedure to build symbolic systems with a product extension.
\end{rem}

\subsection{Building factors of product systems}

 Let $(X,\sigma_X)$ and $(Y,\sigma_Y)$ be two minimal one dimensional shifts and let $\mathcal{A}_X$ and $\mathcal{A}_Y$ be the respective alphabets.

Let $x\in X$ and $y\in Y$. Consider the point ${\bf z}\in (\mathcal{A}_X\times \mathcal{A}_Y)^{\Z^2}$ defined as ${\bf z}_{i,j}=(x_i,y_j)$ for $i,j\in \Z$ and let $Z$ denote the orbit closure of ${\bf z}$ under the shift transformations. Then we can verify that $(Z,\sigma_{(1,0)},\sigma_{(0,1)})$ is isomorphic to the product of $(X,\sigma_X)$ and $(Y,\sigma_Y)$ (and particularly $(Z,\sigma_{(1,0)},\sigma_{(0,1)})$ is a minimal system).

Let $\mathcal{A}$ be an alphabet and let $\varphi\colon\mathcal{A}_X\times \mathcal{A}_Y \to \mathcal{A}$ be a function. We can define $\phi\colon Z\to W\coloneqq \phi(Z)\subseteq \mathcal{A}^{\Z^2}$ such that $\phi(z)_{i,j}=\varphi(z_{i,j})$ for $i,j \in \Z$. Then $(W,\sigma_{(1,0)},\sigma_{(0,1)})$ is a minimal symbolic system with a product extension and we write $W=W(X,Y,\varphi)$ to denote this system. We show that this is the unique way to produce minimal symbolic systems with product extensions.

\begin{prop} \label{SymbolicProduct}
 Let $(W,\sigma_{(1,0)},\sigma_{(0,1)})$ be a minimal symbolic system with a product extension. Then, there exist one dimensional minimal subshifts $(X,\sigma_X)$ and $(Y,\sigma_Y)$ and a map $\varphi\colon\mathcal{A}_X\times \mathcal{A}_Y\to \mathcal{A}_W$ such that $W=W(X,Y,\varphi).$
\end{prop}

\begin{proof}
We recall that $\mathcal{A}_W$ denotes the alphabet of $W$. For $n\in \N$ we let $B_n$ denote $([-n,n]\cap \Z)\times ([-n,n]\cap \Z)$.
Let ${\bf w}=(w_{i,j})_{i,j\in \Z} \in W$. By Proposition \ref{FunctionCoordinate}, the last coordinate in $K^{{\bf w}}_{\sigma_{(1,0)},\sigma_{(0,1)}}(W)$ is a function of the two first coordinates. Since $K^{{\bf w}}_{\sigma_{(1,0)},\sigma_{(0,1)}}(W)$ is a  closed subset of $X^3$ we have that this function is continuous. Hence, there exists $n\in \N$ such that for every $i,j\in \Z$, $w_{i,j}$ is determined by ${\bf w}|_{B_{n}}$, ${\bf w}|_{B_{n}+(i,0)}$ and ${\bf w}|_{B_n+(0,j)}$. Let $\mathcal{A}_X=\{{\bf w}|_{B_n+(i,0)}\colon i \in \Z\}$ and $\mathcal{A}_Y=\{{\bf w}|_{B_n+(0,j)}\colon j \in \Z\}$. Then $\mathcal{A}_X$ and $\mathcal{A}_Y$ are finite alphabets and we can define $\varphi \colon \mathcal{A}_X\times \mathcal{A}_Y\to \mathcal{A}_W$ such that $\varphi({\bf w}|_{B_n+(i,0)},{\bf w}|_{B_n+(0,j)})=w_{i,j}$.

 We recall that since $(W,\sigma_{(1,0)},\sigma_{(0,1)})$ has a product extension, $(K^{w_0}_{\sigma_{(1,0)},\sigma_{(0,1)}}(W),\widehat{\sigma_{(1,0)}},\widehat{\sigma_{(0,1)}})$ is a minimal system.  Let $\phi_1\colon K^{w_0}_{\sigma_{(1,0)},\sigma_{(0,1)}}(W)\to \mathcal{A}_X^{\Z}$ and $\phi_2\colon K^{w_0}_{\sigma_{(1,0)},\sigma_{(0,1)}}(W)\to \mathcal{A}_Y^{\Z}$ defined as $\phi_1(w_1,w_2,w_3)=(w_1|_{B_n+(i,0)})_{i\in \Z}$ and $\phi_2(w_1,w_2,w_3) =(w_2|_{B_n+(0,j)})_{j\in \Z}$. Let $X=\phi_1(W)$ and $Y=\phi_2(W)$. Then $(X,\sigma_X)$ and $(Y,\sigma_Y)$ are two minimal symbolic systems and $W=W(X,Y,\varphi)$.
\end{proof}

The previous proposition says that for a minimal symbolic system $(W,\sigma_{(1,0)},\sigma_{(0,1)})$, having a product extension means that the dynamics can be deduced by looking at the shifts generated by finite blocks in the canonical directions.

\begin{rem}
It was proved in \cite{Mozes} that two dimensional rectangular substitutions are sofic. It was also proved that the product of two one dimensional substitution is a two dimensional substitution and therefore is sofic. Moreover, this product is measurably isomorphic to a shift of finite type. Given Proposition \ref{SymbolicProduct}, the natural question that one can formulate is what properties can be deduced for the subshifts $(X,\sigma_X)$ and $(Y,\sigma_Y)$? For example, what happens with these subshifts when $(W,\sigma_{(1,0)},\sigma_{(0,1)})$ is a two dimensional substitution with a product extension? We do not know the answer to this question.
\end{rem}

\subsection{The Robinson Tiling}

Consider the following set of tiles and their rotations and reflections:

\begin{figure}[H]
\begin{tikzpicture}
\matrix[column sep=0.0cm,row sep=0.0cm,ampersand replacement=\&]{
   \robA \& \TT \&\robB \& \TT \& \robC \& \TT \& \robD \& \TT \& \robE \\
   }
   ;
 \end{tikzpicture}
\caption{ The Robinson Tiles (up to rotation and reflection). The first tile and its rotations are called crosses.} 
\end{figure}
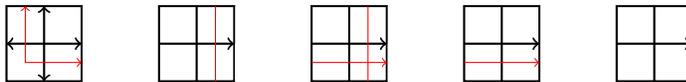

Let $\mathcal{A}$ be the set of the 28 Robinson tiles. Let $Y\subseteq \mathcal{A}^{\Z^2}$ be the subshift defined by the following rules:

  \begin{enumerate}
   \item  The outgoing arrows match with the ingoing arrows;
   \item  There exists $\vec{n}\in \Z^2$ such that there is a cross in every position of the form $\{\vec{n}+(2i,2j) \}$ for $i,j\in \Z$ ( this means that there is a 2-lattice of crosses).
  \end{enumerate}

This system is not minimal but it has a unique minimal subsystem \cite{GJS}. We let $X_{R}$ denote this unique minimal subsystem. Then $(X_{R},\sigma_{(1,0)},\sigma_{(0,1)})$ is a minimal system with commuting transformations $\sigma_{(1,0)}$ and $\sigma_{(0,1)}$ and we call it the {\it minimal Robinson system}. For $n\in \N$ we define {\it supertiles of order $n$} inductively. Supertiles of order 1 correspond to crosses and if we have defined supertiles of order $n$, supertiles of order $n+1$ are constructed putting together 4 supertiles of order $n$ in a consistent way and adding a cross in the middle of them (see Figure \ref{Supertile3}). We remark that supertiles of order $n$ have size $2^n-1$ and they are completely determined by the cross in the middle. Particularly, for every $n\in \N$ there are four supertiles of order $n$. It can be proved \cite{GJS}, \cite{Rad} that for every $x\in X_{R}$, given $n\in \N$, supertiles of order $n$ appear periodically (figure \ref{proofTilefig} illustrates this phenomenon). 

\begin{figure}[H]

\includegraphics[scale=0.37]{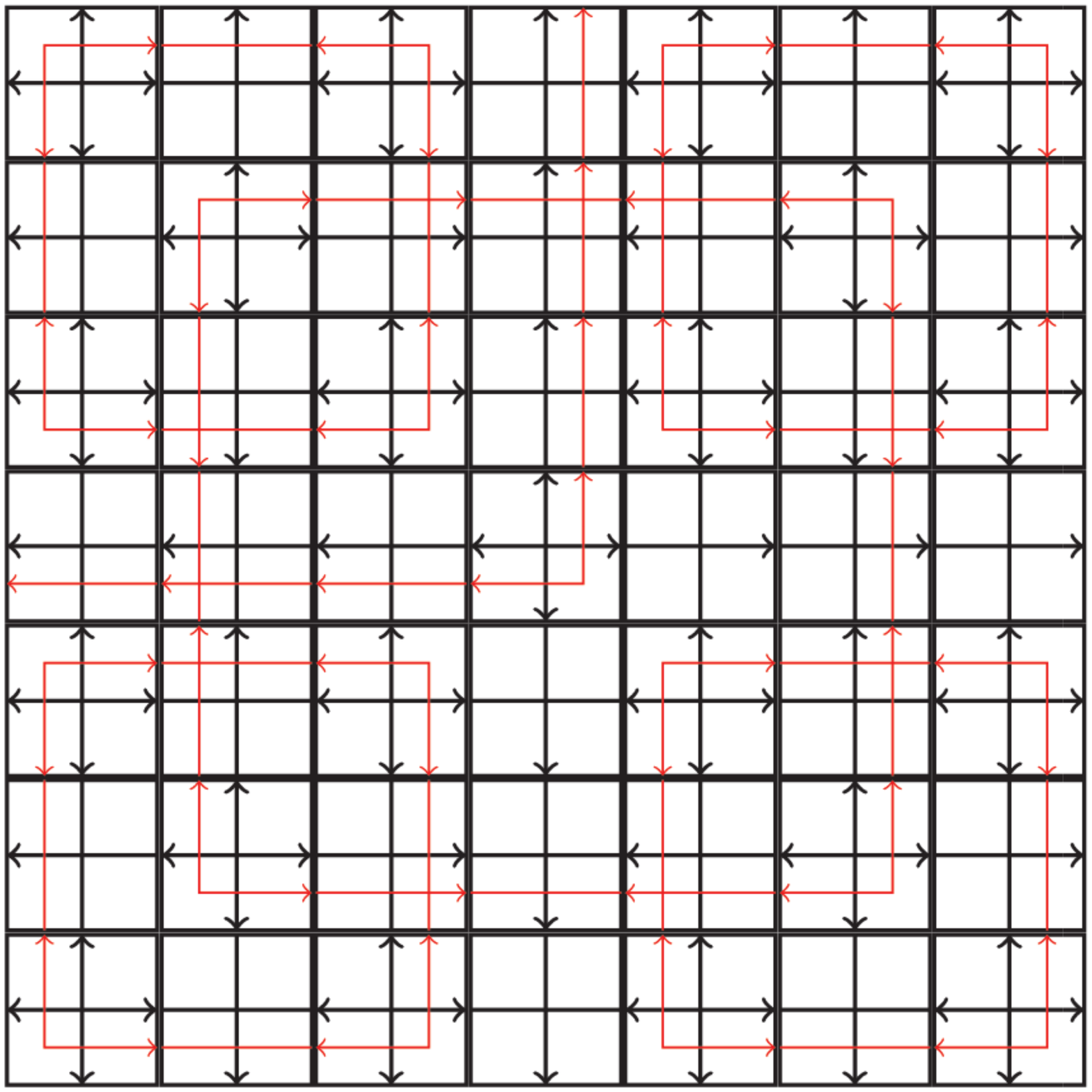}

  \caption{ A supertile of order 3. The four 3x3 squares of the corners are supertiles of order 2.}

 \label{Supertile3}
 \end{figure}

Let $x\in X_{R}$. A horizontal line in $x$ is the restriction of $x$ to a set of the form $\{(i,j_0)\colon i \in \Z\}$ where $j_0\in \Z$. Similarly, a vertical line in $x$ is the restriction of $x$ to a set of the form $\{(i_0,j)\colon j\in \Z\}$ where $i_0\in \Z$. We remark that a line passing through the center of a supertile of order $n$ has only one cross restricted to the supertile. The presence of supertiles of any order, forces the the existence of lines (vertical or horizontal) with at most one cross that are called {\it fault lines}. A point $x\in X_{R}$ can have 0,1 or 2 fault lines. When $x$ is a point with two fault lines, then these lines divide the plane in four quarter planes (one line is horizontal and the other is vertical). On each one of these quarter planes the point is completely determined. The tile in the intersection of two fault lines determines completely the fault lines and therefore this tile determines $x$. See \cite{Rad}, Chapter 1, Section 4 for more details.

Given a point $x\in X_{R}$ and $n\in \N$, supertiles of order $n$ appear periodically, leaving lines between them (which are not periodic). We remark that the center of one of the supertiles of order $n$ determines the distribution of all the supertiles of order $n$. We say that we {\it decompose $x$ into supertiles of order $n$} if we consider the distribution of its supertiles of order $n$, ignoring the lines between them.

Let $B_n\coloneqq ([-2^{n-1},2^{n-1}]\cap \Z) \times ([-2^{n-1},2^{n-1}]\cap \Z)$ be the square of side of size $2^n+1$. Recall that $x|_{B_n}\in \mathcal{A}^{B_n}$ is the restriction of $x$ to $B_n$. Then, looking at $x|_{B_n}$, we can find the center of at least one supertile of order $n$, and therefore we can determine the distribution of supertiles of order $n$ in $x$. We remark that if $x$ and $y$ are points in $X$ such that $x|_{B_n}=y|_{B_n}$, then we can find the same supertile of order $n$ in the same position in $x$ and $y$, and therefore $x$ and $y$ have the same decomposition into tiles of order $n$.

We study the $\mathcal{R}_{\sigma_{(1,0)},\sigma_{(0,1)}}(X_{R})$ relation in the minimal Robinson system. We have:

\begin{prop} \label{R_Robinson}
Let $(X_{R},\sigma_{(1,0)},\sigma_{(0,1)})$ be the minimal Robinson system. Then $(x,y)\in \mathcal{R}_{\sigma_{(1,0)},\sigma_{(0,1)}}$ if and only if they coincide in the complement of its fault lines. Particularly, points which have no fault lines are not related to any point by $\mathcal{R}_{\sigma_{(1,0)},\sigma_{(0,1)}}(X_{R})$.
\end{prop}

\begin{proof}
We start computing the $\mathcal{R}_{\sigma_{(1,0)}}(X_{R})$ relation.  Let $x,y \in \mathcal{R}_{\sigma_{(1,0)}}(X_{R})$ with $x\neq y$ (the case $\mathcal{R}_{\sigma_{(0,1)}}(X_{R})$ is similar). Let $p\in \N$ be such that $x|_{B_p}\neq y|_{B_p}$ and let $x'\in X$, $n,m\in \Z$ and $z\in X_{R}$ with $x'|_{B_p}=x|_{B_p}$, $\sigma_{(1,0)}^nx'|_{B_p}=y|_{B_p}$, $\sigma_{(0,1)}^mx'|_{B_p}=z|_{B_p}$ and $\sigma_{(1,0)}^n\sigma_{(0,1)}^mx'|_{B_p}=z|_{B_p}$. Then $\sigma_{(1,0)}^n\sigma_{(0,1)}^mx'|_{B_p}=\sigma_{(0,1)}^mx'|_{B_p}$ and thus $\sigma_{(1,0)}^n\sigma_{(0,1)}^mx'$ and $\sigma_{(0,1)}^mx'$ have the same decomposition into supertiles of order $p$, which implies that $x$ and $y$ have also the same decomposition. Particularly, the difference between $x$ and $y$ must occur in the lines which are not covered by the supertiles of order $p$ (we remark that these lines have at most one cross). Let $L_p$ be such a line on $x$. For $q$  larger than $p$, we decompose into tiles of order $q$ and we conclude that $L_p$ lies inside $L_q$. Taking the limit in $q$, we deduce that $x$ and $y$ coincide everywhere except in one or two fault lines.

Now suppose that $x$ and $y$ coincide everywhere except in fault lines. For instance, suppose that $x$ and $y$ have two fault lines and let $n\in \N$. We can find $z \in X_{R}$ with no fault lines and $p\in \Z$ such that $z|_{B_n}=x|_{B_n}$ and $\sigma_{(1,0)}^pz|_{B_n}=y_{B_n}$. Then, we can find a supertile of large order containing $z|_{B_n}$ and $\sigma_{(0,1)}^pz|_{B_n}$. Hence, along the horizontal we can find $q\in \Z$ such that $\sigma_{(0,1)}^qz|_{B_n}=\sigma_{(0,1)}^q \sigma_{(1,0)}^p z|_{B_n}$. Since $n$ is arbitrary, we have that $(x,y)\in \mathcal{R}_{\sigma_{(1,0)}}(X_{R})$.
\end{proof}

\begin{figure}[H]
\begin{center}
 \includegraphics[scale=0.37]{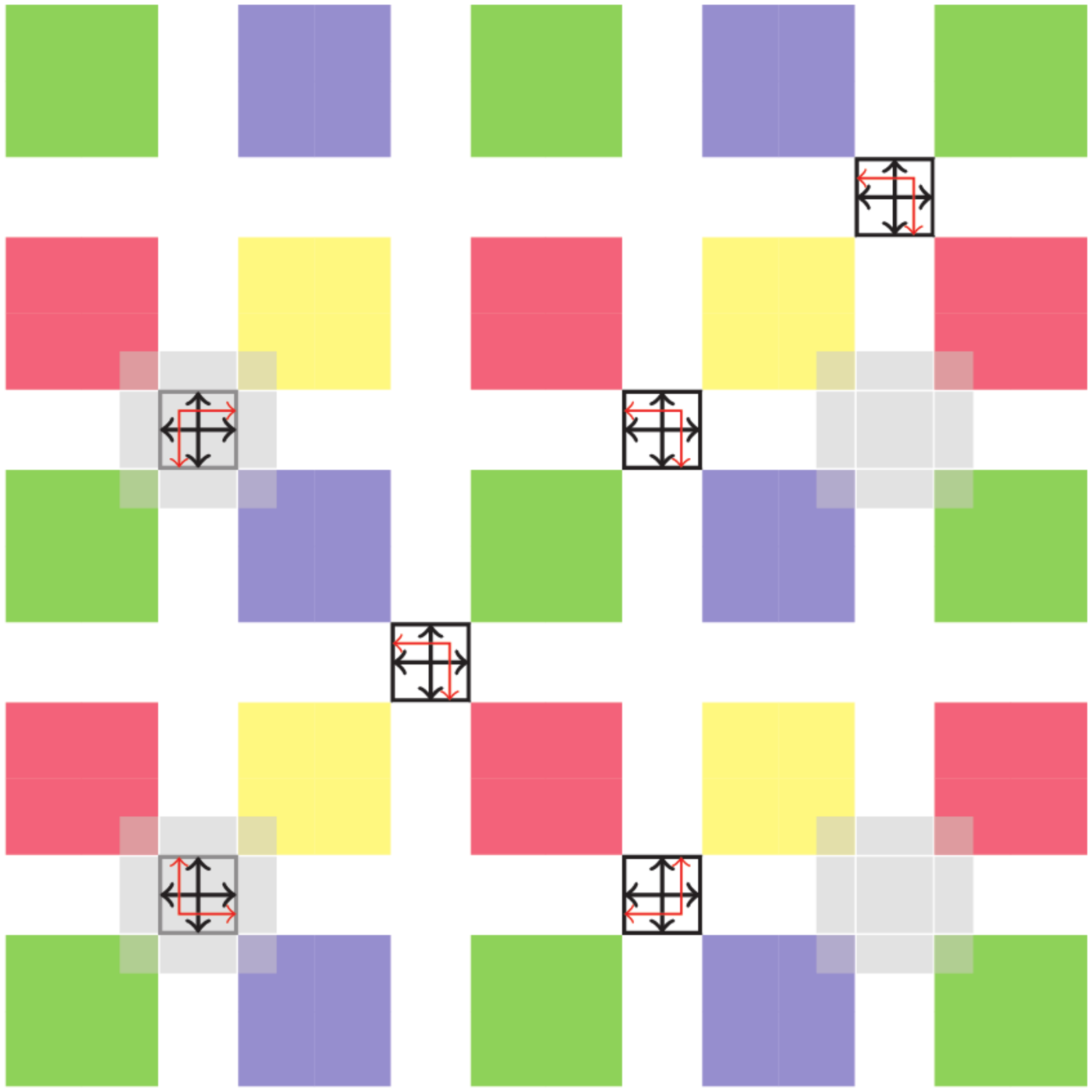}
\end{center}
 \caption{For an arbitrary $n\in\N$, the colored squares represent tiles of order $n$. In this picture we illustrate how points with two fault lines, with different crosses in the middle are related.}

\label{proofTilefig}
 \end{figure}

Let $\pi\colon X\to X/\mathcal{R}_{\sigma_{(1,0)},\sigma_{(0,1)}}(X_{R})$ be the quotient map. By Proposition \ref{R_Robinson} we have that in the minimal Robinson system, we can distinguish three types of fibers for $\pi$: fibers with cardinality  1 (tilings with no fault lines), fibers with cardinality 6 (tilings with one fault line), and fibers with cardinality 28 (tilings with 2 fault lines).

Using Lemma \ref{AutoFiber} with this discussion, we compute the group of automorphisms of $X_{R}$:

\begin{cor}
 The group of automorphisms of the minimal Robinson system is spanned by $\sigma_{(1,0)}$ and $\sigma_{(0,1)}$.
\end{cor}

\begin{proof}
 Clearly, the shift transformations are automorphisms of the system. We claim that there are no other automorphisms. The factor map $\pi\colon X_{R}\to X_{R}/\mathcal{R}_{\sigma_{(1,0)},\sigma_{(0,1)}}(X_{R})$ is almost one to one since points with no fault lines are not related with any point. Consider two fibers $F$ and $F'$ with 28 elements. We remark that in such a fiber, all points have two fault lines and they coincide outside them. It follows that such a fiber is determined only by the position where those two fault lines intersect. Therefore, if $F$ and $F'$ are two fibers with 28 elements, we have that $F'=\sigma_{(1,0)}^n\sigma_{(0,1)}^mF$ for some $n,m\in \Z$. Let $\phi$ be a automorphism of the minimal Robinson system and let $F$ be a fiber with 28 elements. By Lemma \ref{AutoFiber}, $\phi(F)$ is also a fiber with 28 elements. Hence, there exist $n,m\in\Z$ with $\phi(F)=\sigma_{(1,0)}^n\sigma_{(0,1)}^m(F)$. We conclude that the automorphisms $\phi$ and $\sigma_{(1,0)}^n\sigma_{(0,1)}^m$ coincide in one fiber, and by minimality they coincide in all fibers. Particularly, they coincide in fibers with one element, meaning that there exists $x\in X_{R}$ with $\phi(x)=\sigma_{(1,0)}^n\sigma_{(0,1)}^m(x)$. The minimality of system implies that $\phi$ and $\sigma_{(1,0)}^n\sigma_{(0,1)}^m$ are equal.
\end{proof}


\section{$\mathcal{R}_{S,T}(X)$ relation in the distal case}
\subsection{Basic properties}
This section is devoted to the study of the $\mathcal{R}_{S,T}(X)$ relation in the distal case. We do not know if $\mathcal{R}_{S,T}(X)$,  $\mathcal{R}_{S}(X)$ and $\mathcal{R}_{T}(X)$ are equivalence relations in the general setting. However,
we have a complete description of these relations in the distal case.

Recall that a topological dynamical system $(X,G_{0})$ is {\it distal} if $x\neq y$ implies that $$\inf_{g \in G_{0}} d(gx,gy)>0.$$

Distal systems have many interesting properties (see \cite{Aus}, chapters 5 and 7). We recall some of them:
\begin{thm} \label{distal}\quad
 \begin{enumerate}
  \item The Cartesian product of distal systems is distal;
  \item Distality is preserved by taking factors and subsystems;
  \item A distal system is minimal if and only if it is transitive;
   \item If $(X,G_0)$ is distal and $G_0'$ is a subgroup of $G_0$, then $(X,G_0')$ is distal.
 \end{enumerate}
\end{thm}

The main property about distality is that it implies that cubes have the following transitivity property:

\begin{lem}\label{glue} Let $(X,S,T)$ be a distal minimal system with commuting transformations $S$ and $T$. Suppose that $R$ is either $S$ or is $T$. Then

\begin{enumerate}
 \item If $(x,y),(y,z)\in \Q_R(X)$, then $(x,z)\in \Q_R(X)$;
 \item If $(a_{1},b_{1},a_{2},b_{2}),(a_{2},b_{2},a_{3},b_{3})\in \Q_{S,T}(X)$, then $(a_{1},b_{1},a_{3},b_{3})\in \Q_{S,T}(X)$.
\end{enumerate}

\end{lem}
\begin{proof}
We only prove (1) since the proof of (2) is similar.
 Let $(x,y) ,(y,z) \in \Q_R(X)$. Pick any $a\in X$. Then $(a,a)\in\Q_{R}(X)$. By Proposition \ref{Q_STMinimal}, there exists a sequence $(g_n)_{n\in \N}=((g_n',g_n''))_{n\in\N}$ in $\G_R$ such that $g_n(x,y)=(g_n' x, g_n'' y)\to (a,a)$, where $\G_R$ is the group generated by $\id\times R$ and $g\times g, g\in G$. We can assume (by taking a subsequence) that $g_n''z\to u$ and thus $(g_n''y,g_n''z)\to (a,u) \in \Q_R(X)$. Since $(g_n',g_n'')(x,z)\to(a,u)$, by distality we have that $(x,z)$ is in the closed orbit of $(a,u)$ and thus $(x,z)\in \Q_R(X)$.
\end{proof}

\begin{rem}
  It is worth noting that the gluing lemma fails in the non-distal case, even if $S=T$ (see \cite{TuYe} for an example).
\end{rem}

The following proposition gives equivalent definitions of $\mathcal{R}_{S,T}(X)$ in the distal case:

\begin{prop}\label{equ} Let $(X,S,T)$ be a distal system with commuting transformations $S$ and $T$. Suppose $x,y\in X$. The following are equivalent:

\begin{enumerate} \label{RPST}
\item $(x,y,y,y)\in \Q_{S,T}(X)$;

\item There exists $a,b,c\in X$ such that $(x,a,b,c),(y,a,b,c)\in\Q_{S,T}(X)$;

\item For every $a,b,c\in X$, if $(x,a,b,c)\in \Q_{S,T}(X)$, then $(y,a,b,c)\in \Q_{S,T}(X)$;

\item $(x,y)\in\mathcal{R}_{S,T}(X)$;

\item $(x,y)\in\mathcal{R}_{S}(X)$;

\item $(x,y)\in\mathcal{R}_{T}(X)$.
\end{enumerate}
Particularly, $\mathcal{R}_{S}(X)=\mathcal{R}_{T}(X)=\mathcal{R}_{S,T}(X)$.
\end{prop}
\begin{rem}
  This proposition shows that in the distal case, the relation $\mathcal{R}_{S,T}(X)$ coincides with the $\RP^{[2]}_{T}$ relation defined in
  ~\cite{HKM},~\cite{HSY},~\cite{HSY2} and ~\cite{SY} when $S=T$.
\end{rem}

\begin{proof}

  (1)$\Rightarrow$(3). Suppose that $(x,a,b,c)\in \Q_{S,T}(X)$ for some $a,b,c\in X$. By (3),(4) and (5) of Proposition \ref{sym}, $(x,a,b,c)\in \Q_{S,T}(X)$ implies that $(a,x,a,x)\in \Q_{S,T}(X)$, and $(x,y,y,y)\in \Q_{S,T}(X)$ implies that $(x,x,y,x)\in \Q_{T,S}(X)$. By Lemma \ref{glue}, $(a,x,a,x),(x,x,y,x)\in \Q_{S,T}(X)$ implies that $(x,a,y,a)\in \Q_{S,T}(X)$. Again by Lemma \ref{glue}, $(x,a,b,c),(x,a,y,a)\in \Q_{S,T}(X)$ implies that $(b,c,y,a)\in \Q_{S,T}(X)$ and thus $(y,a,b,c)\in \Q_{S,T}(X)$.

  (3)$\Rightarrow$(2). Obvious.

  (2)$\Rightarrow$(1). Suppose that $(x,a,b,c),(y,a,b,c)\in\Q_{S,T}(X)$ for some $a,b,c\in X$. Then $(b,c,y,a)\in\Q_{S,T}(X)$. By Lemma \ref{glue}, $(x,a,y,a)\in\Q_{S,T}(X)$. By (4) and (5) of Proposition \ref{sym}, $(y,a,y,a)\in\Q_{S,T}(X)$. Hence $(x,y,a,a),(y,y,a,a)\in\Q_{S,T}(X)$ and $(a,a,y,y)\in\Q_{S,T}(X)$. By Lemma \ref{glue}, $(x,y,y,y)\in\Q_{S,T}(X)$.

  (1)$\Rightarrow$(4). Take $a=y$ and $b=y$.

  (4)$\Rightarrow$(5) and (4)$\Rightarrow$(6) are obvious from the definition.

  (5)$\Rightarrow$(1). Suppose $(x,y,a,a)\in\Q_{S,T}(X)$ for some $a\in X$. By (4) and (5) of Proposition \ref{sym}, $(y,y,a,a)\in\Q_{S,T}(X)$. By Lemma \ref{glue}, $(x,y,y,y)\in\Q_{T,S}(X)$ and thus $(x,y,y,y)\in\Q_{S,T}(X)$.

  (6)$\Rightarrow$(1). Similar to (4)$\Rightarrow$(2).
\end{proof}

We can now prove that $\mathcal{R}_{S,T}(X)$ is an equivalence relation in the distal setting:

\begin{thm} \label{Equivalence}
  Let $(X,S,T)$ be a distal system with commuting transformations $S$ and $T$. Then
$\Q_S(X)$, $\Q_T(X)$ and $\mathcal{R}_{S,T}(X)$ are closed equivalence relations on $X$. 
\end{thm}
\begin{proof}
It suffices to prove the transitivity of $\mathcal{R}_{S,T}(X)$. Let $(x,y), (y,z)\in \mathcal{R}_{S,T}(X)$. Since $(y,z,z,z)$ and $(x,y)\in \mathcal{R}_{S,T}(X)$, by (4) of Proposition \ref{RPST}, we have that $(x,z,z,z)\in \Q_{S,T}(X)$ and thus $(x,z)\in \mathcal{R}_{S,T}(X).$
\end{proof}

We also have the following property in the distal setting, which allows us to lift an $(S,T)$-regionally proximal pair in a system to a pair
in an extension system:

\begin{prop} \label{LiftingProperty}
 Let $\pi\colon Y\to X$ be a factor map between systems $(Y,S,T)$ and $(X,S,T)$ with commuting transformations $S$ and $T$. If $(X,S,T)$ is distal, then $\pi\times \pi (\mathcal{R}_{S,T}(Y))=\mathcal{R}_{S,T}(X)$.
\end{prop}

\begin{proof}
 The proof is similar to Theorem 6.4 of \cite{SY}. Let $(x_1,x_2)\in \mathcal{R}_{S,T}(X)$. Then there exist a sequence $(x_i)_{i\in \N} \in X$ and two sequences $(n_i)_{i\in \N},(m_i)_{i\in \N}$ in $\Z$ such that  $$(x_i,S^{n_i}x_i,T^{m_i}x_i,S^{n_i}T^{m_i}x_i) \to (x_1,x_1,x_1,x_2).$$

 Let $(y_i)_{i\in \N}$ in $Y$ be such that $\pi(y_i)=x_i$. By compactness we can assume that $y_i\to y_1$, $S^{n_i}y_i\to a$, $T^{m_i}y_i\to b$ and $S^{n_i}T^{m_i}y_i\to c$. Then $(y_1,a,b,c)\in \Q_{S,T}(Y)$ and $\pi^{4}(y_1,a,b,c)=(x_1,x_1,x_1,x_2)$. Particularly, $(y_1,a)\in \Q_S(Y)$. By minimality we can find $g_i \in G$ and $p_i$ such that $(g_i y_1,g_iS^{p_i}a)\to (y_1,y_1)$. We can assume that $g_ib\to b'$ and $g_iS^{p_i}c\to c'$, so that $(y_1,y_1,b',c')\in \Q_{S,T}(Y)$
and $\pi^{4}(y_1,y_1,b',c')=(x_1,x_1,x_1,x_2')$, where $x_2'=\lim g_iS^{p_i}x_2$. Recall that $(x_1,x_2')\in \mathcal{O}_{G^{\Delta}}(x_1,x_2)$, where $G^{\Delta}=\{g\times g\colon g\in G\}$. Since $(y_1,b')\in \Q_T(Y)$, we can find $(g_i')_{i\in \N}$ in $G$ and $(q_i)_{i\in \N}$ in $\Z$ such that $(g_i'y_1,g_i'T^{q_i}b')\to (y_1,y_1)$. We can assume without loss of generality that $g_i'T^{q_i}c'\to c''$ so that $(y_1,y_1,y_1,c'') \in \Q_{S,T}(Y)$ and $\pi^{4}(y_1,y_1,y_1,c'')=(x_1,x_1,x_1,x_2'')$, where $x_2''=\lim g_i'T^{q_i}x_2'$. Recall that $(x_1,x_2'')\in \mathcal{O}_{G^{\Delta}}(x_1,x_2')$. So $(x_1,x_2'')\in \mathcal{O}_{G^{\Delta}}(x_1,x_2)$. By distality, this orbit is minimal and thus we can find $(g_i'')_{i\in \N}$ in $G$ such that $(g_i''x_1,g_i''x_2'')\to (x_1,x_2)$. We assume without loss of generality that $g_i''y_1\to y_1'$ and $g_i''c''\to y_2'$. Then $(y_1',y_1',y_1',y_2')\in \Q_{S,T}(Y)$ and $\pi^{4}(y_1',y_1',y_1',y_2')=(x_1,x_1,x_1,x_2)$. Particularly $(y_1',y_2')\in \mathcal{R}_{S,T}(Y)$ and $\pi\times\pi(y_1',y_2')=(x_1,x_2)$.
\end{proof}

These results allow us to conclude that cubes structures characterize factors with product extensions:

\begin{thm} \label{ThmDistalCas}
 Let $(X,S,T)$ be a minimal distal system with commuting transformations $S$ and $T$. Then
 \begin{enumerate}
\item $(X/\mathcal{R}_{S,T}(X),S,T)$ has a product extension, where $X/\mathcal{R}_{S,T}(X)$ is the quotient of $X$ under the equivalence relation $\mathcal{R}_{S,T}(X)$. Moreover, it is the maximal factor with this property, meaning that any other factor of $X$ with a product extension factorizes through it;
\item For any magic extension $(\bold{K}^{x_0}_{S,T},\widehat{S},\widehat{T})$, $(\bold{K}^{x_0}_{S,T}/\mathcal{R}_{\widehat{S},\widehat{T}}(\bold{K}^{x_0}_{S,T}),\widehat{S},\widehat{T})$ is a product system. Moreover, both $(\bold{K}^{x_0}_{S,T},\widehat{S},\widehat{T})$ and $(\bold{K}^{x_0}_{S,T}/\mathcal{R}_{\widehat{S},\widehat{T}}(\bold{K}^{x_0}_{S,T}))$ are distal systems.
\end{enumerate}

\end{thm}

We have the following commutative diagram:
\begin{center}

 \begin{tikzpicture}
 \matrix (m) [matrix of math nodes,row sep=3em,column sep=4em,minimum width=2em]
  {
     {}(\bold{K}^{x_0}_{S,T},\widehat{S},\widehat{T}) & (X,S,T)\\
     (\bold{K}^{x_0}_{S,T}/\mathcal{R}_{\widehat{S},\widehat{T}}(\bold{K}^{x_0}_{S,T}),\widehat{S},\widehat{T}) & (X/\mathcal{R}_{S,T}(X),S,T) \\} ;

      \path[-stealth]
    (m-1-1) edge node [left] {}  (m-2-1)
            edge node [below] {} (m-1-2)
    (m-2-1) edge node [below] {} (m-2-2)
    (m-1-2) edge node [right] {} (m-2-2) ;
     \end{tikzpicture}

\end{center}

\begin{proof}
  We remark that if $(Z,S,T)$ is a factor of $(X,S,T)$ with a product extension, then $\pi\times \pi(\mathcal{R}_{S,T}(X))=\mathcal{R}_{S,T}(Z)=\Delta_X$, meaning that there exists a factor map from $(X/\mathcal{R}_{S,T}(X),S,T)$ to $(Y,S,T)$. It remains to prove that $X/\mathcal{R}_{S,T}(X)$ has a product extension. To see this, let $\pi$ be the quotient map $X\to X/\mathcal{R}_{S,T}(X)$ and let $(y_1,y_2)\in \mathcal{R}_{S,T}(X/\mathcal{R}_{S,T}(X))$. By Proposition \ref{LiftingProperty}, there exists $(x_1,x_2)\in \mathcal{R}_{S,T}(X)$ with $\pi(x_1)=y_1$ and $\pi(x_2)=y_2$. Since $(x_1,x_2)\in \mathcal{R}_{S,T}(X)$, $y_{1}=\pi(x_{1})=\pi(x_{2})=y_{2}$. So $\mathcal{R}_{S,T}(X/\mathcal{R}_{S,T}(X))$ coincides with the diagonal. By Theorem \ref{R-ProductExtension}, $(X/\mathcal{R}_{S,T}(X),S,T)$ has a product extension. This proves (1).

  We now prove that the factor of the magic extension is actually a product system. By Theorem \ref{Equivalence}, we have that $\Q_{\widehat{S}}(\bold{K}^{x_0}_{S,T}), \Q_{\widehat{T}}(\bold{K}^{x_0}_{S,T})$ are equivalence relations and by Theorem \ref{MagicExtension} and Proposition \ref{RPST}, we have that $\Q_{\widehat{S}}(\bold{K}^{x_0}_{S,T})\cap \Q_{\widehat{T}}(\bold{K}^{x_0}_{S,T})=\mathcal{R}_{S,T}(\bold{K}^{x_0}_{S,T})$.  Consequently $(\bold{K}^{x_0}_{S,T}/\mathcal{R}_{\widehat{S},\widehat{T}}(\bold{K}^{x_0}_{S,T}),\widehat{S},\widehat{T})$ is isomorphic to $(\bold{K}^{x_0}_{S,T}/\Q_{\widehat{T}}(\bold{K}^{x_0}_{S,T})\times \bold{K}^{x_0}_{S,T}/ \Q_{\widehat{S}}(\bold{K}^{x_0}_{S,T}), \widehat{S}\times \id, \id \times \widehat{T})$, which is a product system.

  Since $(X,S,T)$ is distal, the distality of $(\bold{K}^{x_0}_{S,T},\widehat{S},\widehat{T})$ and $(\bold{K}^{x_0}_{S,T}/\mathcal{R}_{S,T}(\bold{K}^{x_0}_{S,T}),\widehat{S},\widehat{T})$ follows easily from Theorem \ref{distal}.
\end{proof}

\subsection{Further remarks: The $\mathcal{R}_{S,T}(X)$ strong relation}
Let $(X,S,T)$ be a system with commuting transformations $S$ and $T$. We say that $x$ and $y$ are {\it strongly $\mathcal{R}_{S,T}(X)$-related} if there exist $a\in X$ and two sequences $(n_i)_{i\in \N}$ and $(m_i)_{i\in \N}$ in $\Z$ such that $(x,y,a,a)=\lim\limits_{i\to \infty} (x,S^{n_i}x,T^{m_i}x,S^{n_i}T^{m_i}x)$, and there exist $b\in X$ and two sequences $(n'_i)_{i\in \N}$ and $(m'_i)_{i\in \N}$ in $\Z$ such that $(x,b,y,b)=\lim\limits_{i\to \infty} (x,S^{n'_i}x,T^{m'_i}x,S^{n'_i}T^{m'_i}x)$.

It is a classical result that when $S=T$, the $\mathcal{R}_{T,T}(X)$ relation coincides with the strong one (see \cite{Aus}, Chap 9). We show that this is not true in the commuting case even in the distal case, and give a counter example of commuting rotations in the Heisenberg group. We refer to \cite{AGH} and \cite{L} for general references about nilrotations.

Let $H=\mathbb{R}^{3}$ be the group with the multiplication given by $(a,b,c)\cdot(a',b',c')=(a+a',b+b',c+c'+ab')$ for all $(a,b,c),(a',b',c')\in H$. Let $H_2$ be the subgroup spanned by $\{ghg^{-1}h^{-1}\colon g,h\in H\}$. By a direct computation we have that $H_2=\{(0,0,c)\colon c\in \mathbb{R}\}$ and thus $H_2$ is central in $H$. Therefore $H$ is a 2-step nilpotent Lie group and $\Gamma=\mathbb{T}^3$ is a cocompact subgroup, meaning that $X_{H}\coloneqq H/\Gamma$ is a compact space. $X_{H}$ is called the {\it Heisenberg manifold}. Note that $\mathbb{T}^3$ is a fundamental domain of $X_{H}$.

\begin{lem}\label{funda}
The map $\Phi\colon X_{H}\to \mathbb{T}^3$ given by
$$\Phi((a,b,c)\Gamma)=(\{a\},\{b\},\{c-a\lfloor{b}\rfloor \})$$
is a well-defined homomorphism between $X_{H}$ and $\mathbb{T}^3$. Here $\lfloor{x}\rfloor$ is the largest integer which does not exceed $x$, $\{x\}=x-\lfloor x\rfloor$, and $\mathbb{T}^3$ is viewed as $[0,1)^{3}$ in this map. Moreover, $(a,b,c)\Gamma=(\{a\},\{b\},\{c-a\lfloor{b}\rfloor \})\Gamma$ for all $a,b,c\in\mathbb{R}$.
\end{lem}
\begin{proof}
  It suffices to show that $(a,b,c)\Gamma=(a',b',c')\Gamma$ if and only if $(\{a\},\{b\},\{c-a\lfloor{b}\rfloor \})=(\{a'\},\{b'\},\{c'-a'\lfloor{b'}\rfloor \})$. If $(a,b,c)\Gamma=(a',b',c')\Gamma$, there exists $(x,y,z)\in\Gamma$ such that $(a',b',c')=(a,b,c)\cdot(x,y,z)=(x+a,y+b,z+c+ay)$. therefore,
$$x=a'-a, y=b'-b, z=c'-c-a(b'-b).$$
Since $x,y\in\Z$, we have that $\{a\}=\{a'\},\{b\}=\{b'\}$. So $b-b'=\lfloor{b}\rfloor-\lfloor{b'}\rfloor$. Then
$$(c'-a'\lfloor{b'}\rfloor)-(c-a\lfloor{b}\rfloor)=(c'-c-a(b'-b))-(a'-a)\lfloor{b}\rfloor=z-x\lfloor{b}\rfloor\in\Z.$$
So $(\{a\},\{b\},\{c-a\lfloor{b}\rfloor \})=(\{a'\},\{b'\},\{c'-a'\lfloor{b'}\rfloor \})$.

Conversely, if $(\{a\},\{b\},\{c-a\lfloor{b}\rfloor \})=(\{a'\},\{b'\},\{c'-a'\lfloor{b'}\rfloor \})$, suppose that
$$x=a'-a, y=b'-b, z=c'-c-a(b'-b).$$
Then $(a',b',c')=(a,b,c)\cdot(x,y,z)$. It remains to show that $(x,y,z)\in\Gamma$. Since $\{a\}=\{a'\},\{b\}=\{b'\}$, we have that $x,y\in\Z$ and $b-b'=\lfloor{b}\rfloor-\lfloor{b'}\rfloor$. Then
$$(c'-a'\lfloor{b'}\rfloor)-(c-a\lfloor{b}\rfloor)=(c'-c-a(b'-b))-(a'-a)\lfloor{b}\rfloor=z-x\lfloor{b}\rfloor\in\Z$$
implies that $z\in\Z$.

The claim that $(a,b,c)\Gamma=(\{a\},\{b\},\{c-a\lfloor{b}\rfloor \})\Gamma$ for all $a,b,c\in\mathbb{R}$ is straightforward.
\end{proof}

Let $\alpha\in \mathbb{R}$ be such that $1,\alpha,\alpha^{-1}$ are linearly independent over $\mathbb{Q}$. Let $s=(\alpha,0,0)$ and $t=(0,\alpha^{-1},\alpha)$. These two elements induce two transformations $S,T\colon X_{H}\to X_{H}$ given by
 $$S(h\Gamma)=sh\Gamma, T(h\Gamma)=th\Gamma, \forall h\in H.$$

 \begin{lem}
 Let $X_{H},S,T$ be defined as above. Then $(X_{H},S,T)$ is a minimal distal system with commuting transformations $S$ and $T$.
 \end{lem}
 \begin{proof}
   We have that $st=(\alpha,\alpha^{-1},\alpha+1)$ and $ts=(\alpha,\alpha^{-1},\alpha)$ and by a direct computation we have that they induce the same action on $X_{H}$. Therefore $ST=TS$.

   It is classical that a rotation on a nilmanifold is distal \cite{AGH} and it is minimal if and only if the rotation induced on its maximal equicontinuous factor is minimal. Moreover, the maximal equicontinuous factor is given by the projection on $H/H_2\Gamma$ which in our case is nothing but the projection in $\mathbb{T}^2$ (the first two coordinates). See \cite{L} for a general reference on nilrotations.

 Since $ST(h\Gamma)=(\alpha,\alpha^{-1},\alpha)\cdot h\Gamma$ for all $h\in H$, we have that the induced rotation on $\mathbb{T}^2$ is given by the element $(\alpha,\alpha^{-1})$. Since $1$, $\alpha$ and $\alpha^{-1}$ are linearly independent over $\mathbb{Q}$, by the Kronecker Theorem we have that this is a minimal rotation. We conclude that $(X_{H},ST)$ is minimal which clearly implies that $(X_{H},S,T)$ is minimal.
  
 \end{proof}

 In this example, we show that the relation $\mathcal{R}_{S,T}(X)$ is different from the strong one:

\begin{prop}
 On the Heisenberg system $(X_{H},S,T)$, we have that
 $$\mathcal{R}_{S,T}(X_{H})=\Bigl\{\bigl((a,b,c)\Gamma,(a,b,c')\Gamma\bigr)\in X_{H}\times X_{H}\colon a,b,c,c'\in\mathbb{R}\Bigr\}.$$
However, for any $c\in\mathbb{R}\backslash\mathbb{Z}$, $\Gamma$ and $(0,0,c)\Gamma$ are not strongly $\mathcal{R}_{S,T}(X_{H})$-related.
 
\end{prop}

\begin{proof}
Suppose that $((a,b,c)\Gamma,(a',b',c')\Gamma)\in\mathcal{R}_{S,T}(X_{H})$. Then $((a,b,c)\Gamma,(a',b',c')\Gamma)\in\mathcal{R}_{T}(X_{H})$. Projecting to the first coordinate, we have that
$(\{a\},v,\{a'\},v)\in\Q_{\overline{S},id}(\mathbb{T})$ for some $v\in\mathbb{T}$, where in the system $(\mathbb{T},\overline{S},\id)$, $\overline{S}x=x+\alpha$ for all $x\in\mathbb{T}$ (we regard $\mathbb{T}$ as [0,1)). Since the second transformation is identity, we have that $\{a\}=\{a'\}$. Similarly, $\{b\}=\{b'\}$.
So in order to prove the first statement, it suffices to show that $((a,b,c)\Gamma,(a,b,c')\Gamma)\in\mathcal{R}_{S,T}(X_{H})$ for all $a,b,c,c'\in\mathbb{R}$. Since $(X_{H},S,T)$ is minimal, exist a sequence $(g_{i})_{i\in\N}$ in $G$ and a sequence $(c_{i})_{i\in\N}$ in $\mathbb{R}$ such that

$$\lim_{i\to\infty}g_{i}((0,0,0)\Gamma)=(a,b,c),\quad \lim_{i\to\infty}g_{i}((0,0,c_{i})\Gamma)=(a,b,c').$$

Since $\mathcal{R}_{S,T}(X_{H})$ is closed and invariant under $g\times g, g\in G$, it
then suffices to show that $\Gamma$ and $(0,0,c)\Gamma$ are $\mathcal{R}_{S,T}(X_{H})$-related for all $c\in\mathbb{R}$. Fix $\epsilon>0$. Let $n_i\to +\infty$ be such that $|\{n_i\alpha\}|<\e$ and $\frac{c}{n_i\alpha}<\e$. Let $x_i=(0,\frac{c}{n_i\alpha},0)\Gamma$. Then $d(x_i,\Gamma)<\e$ and by Lemma \ref{funda}, we have that

$$S^{n_i}x_i=(n_i\alpha,\frac{c}{n_i\alpha},c)\Gamma=(\{n_i\alpha\},\frac{c}{n_i\alpha},\{c-n_i\alpha\lfloor \frac{c}{n_i\alpha} \rfloor \})\Gamma=(\{n_i\alpha\},\frac{c}{n_i\alpha},c)\Gamma.$$

So $d(S^{n_i}x_i,(0,0,c)\Gamma)<2\e$. We also have that $d(S^{n_i}(0,0,c)\Gamma, (0,0,c)\Gamma)<\e$. Let $\delta>0$ be such that if $d(h\Gamma,h'\Gamma)<\delta$, then $d(S^{n_i}h\Gamma,S^{n_i}h'\Gamma)<\e$. Since the rotation on $(\alpha,\alpha^{-1})$ is minimal in $\mathbb{T}^2$, we can find $m_i$ large enough such that $0<\{m_i\alpha\}+\frac{c}{n_i}<\delta$ and $|\{m_i\alpha^{-1}\}-c|<\delta$. Hence, $d(T^{m_i}x_i,(0,0,c)\Gamma)<\delta$ and thus $d(S^{n_i}T^{m_i}x_i,(0,0,c)\Gamma)<2\e$. It follows that for large enough $i$, the distance between $(\Gamma,(0,0,c)\Gamma,(0,0,c)\Gamma,(0,0,c)\Gamma)$ and $(x_i,S^{n_i}x_i,T^{m_i}x_i,S^{n_i}T^{m_i}x_i)$ is less than $6\epsilon$. Since $\epsilon$ is arbitrary, we get that $$(\Gamma,(0,0,c)\Gamma,(0,0,c)\Gamma,(0,0,c)\Gamma) \in \Q_{S,T}(X_{H})$$ and thus $\Gamma$ and $(0,0,c)\Gamma$ are $\mathcal{R}_{S,T}(X_{H})$-related. This finishes the proof of the first statement.

For the second statement,
let $h=(h_1,h_2,h_3)\in H$ with $h_i\in [0,1)$ for $i=1,2,3$. We remark that $S^n\Gamma=(n\alpha,0,0)\Gamma=(\{n\alpha\},0,0)\Gamma$. So if $(\Gamma,h\Gamma)$ are $\mathcal{R}_{S,T}(X_{H})$-strongly related, then $h_2=h_3=0$. Hence for $c\in (0,1)$, $\Gamma$ and $(0,0,c)\Gamma$ are not $\mathcal{R}_{S,T}(X_{H})$-strongly related.
\end{proof}

\subsection{A Strong form of the $\mathcal{R}_{S,T}(X)$ relation}

We say that $(x_1,x_2)\in X\times X$ are \emph{$\mathcal{R}_{S,T}^{\ast}(X)$-related} if there exist $(n_i)_{i\in \N}$ and $(m_i)_{i\in \N}$ sequences in $\Z$ such that $$(x_1,S^{n_i}x_1,T^{m_i}x_1,S^{n_i}T^{m_i}x_1)\to (x_1,x_1,x_1,x_2).$$ Obviously, $\mathcal{R}_{S,T}^{\ast}(X)\subseteq \mathcal{R}_{S,T}(X)$.

In this subsection, we prove that the relation generated by $\mathcal{R}_{S,T}^{\ast}(X)$ coincides with the $\mathcal{R}_{S,T}(X)$ relation. We start with some lemmas:

\begin{rem}
 It is shown in \cite{TuYe} that, even in the case $S=T$,  the relation generated by $\mathcal{R}_{S,T}^{\ast}(X)$ may not coincide with the $\mathcal{R}_{S,T}(X)$ relation in the non-distal setting. In fact, there exists a system with $\mathcal{R}_{T,T}^{\ast}=\Delta_X\neq \mathcal{R}_{T,T}$.
\end{rem}

\begin{lem} \label{RPTrivial}
 Let $(X,S,T)$ be a minimal distal system with commuting transformations $S$ and $T$. Then $\mathcal{R}_{S,T}(X)=\Delta_X$ if and only if $\mathcal{R}_{S,T}^{\ast}(X)=\Delta_X$.
\end{lem}

\begin{proof}
 We only prove the non-trivial direction. Suppose that $\mathcal{R}_{S,T}^{\ast}(X)$ coincides with the diagonal. Fix $x_0 \in X$ and consider the system $(\bold{K}^{x_0}_{S,T},\widehat{S},\widehat{T})$. Let $\mathcal{R}_{\widehat{S},\widehat{T}}[(x_0,x_0,x_0)]$ be the set of points that are $\mathcal{R}_{\widehat{S},\widehat{T}}$ related with $(x_0,x_0,x_0)$. Pick $(x_1,x_2,x_3) \in \mathcal{R}_{\widehat{S},\widehat{T}}[(x_0,x_0,x_0)]$. By definition, we have that $x_1=x_2=x_0$. Hence $(x_0,x_0,x_3)\in \bold{K}^{x_0}_{S,T} $ and thus $(x_0,x_3)$ belongs to $\mathcal{R}_{S,T}^{\ast}(X)$. We conclude that $\# \mathcal{R}_{\widehat{S},\widehat{T}}[(x_0,x_0,x_0)]=1$. By distality and minimality, the same property holds for every point in $\bold{K}^{x_0}_{S,T}$ and thus $\mathcal{R}_{\widehat{S},\widehat{T}}(\bold{K}^{x_0}_{S,T})$ coincides with the diagonal relation. Particularly, $(\bold{K}^{x_0}_{S,T},\widehat{S},\widehat{T})$ has a product extension and consequently so has $(X,S,T)$. This is equivalent to saying that $\mathcal{R}_{S,T}(X)=\Delta_X$.
\end{proof}

Let $\mathcal{R}(X)$ be the relation generated by $\mathcal{R}_{S,T}^{\ast}(X)$. We have:
\begin{lem} \label{Relation}
 Let $\pi\colon Y\to X$ be the factor map between two minimal distal systems $(Y,S,T)$ and $(X,S,T)$ with commuting transformations $S$ and $T$. Then $\pi\times \pi (\mathcal{R}(Y)) \supseteq \mathcal{R}_{S,T}^{\ast}(X)$.

\end{lem}

\begin{proof}
 Similar to the proof of Proposition \ref{LiftingProperty}.
\end{proof}

We can now prove the main property of this subsection:

\begin{prop}
 Let $(X,S,T)$ be a distal minimal system with commuting transformations $S$ and $T$. Then $\mathcal{R}(X)=\mathcal{R}_{S,T}(X)$.
\end{prop}

\begin{proof}
 We only need to prove that $\mathcal{R}_{S,T}(X)\subseteq \mathcal{R}(X)$. Let $\pi\colon X\to X/\mathcal{R}(X)$ be the projection map. By Lemma \ref{Relation}, $\Delta_X=\pi \times \pi (\mathcal{R}(X))\supseteq \mathcal{R}_{S,T}^{\ast}(X/\mathcal{R}(X))$. By Lemma \ref{RPTrivial}, $\mathcal{R}_{S,T}(X/\mathcal{R}(X))=\Delta_X$ and then $(X/\mathcal{R}(X),S,T)$ has a product extension. By Theorem \ref{ThmDistalCas} $(X/\mathcal{R}_{S,T}(X),S,T)$ is the maximal factor with this property and therefore $\mathcal{R}_{S,T}(X)\subseteq \mathcal{R}(X)$.
\end{proof}

.

\section{Properties of systems with product extensions}
In this section, we study the properties of systems which have a product extension. We characterize them in terms of their enveloping semigroup and we study the class of systems which are disjoint from them. Also, in the distal case we study properties of recurrence and topological complexity.

\subsection{The enveloping semigroup of systems with a product extension}

Let $(X,S,T)$ be a system with commuting transformations $S$ and $T$, and let $E(X,S)$ and $E(X,T)$ be the enveloping semigroups associated to the systems $(X,S)$ and $(X,T)$ respectively. Hence $E(X,S)$ and $E(X,T)$ are subsemigroups of $E(X,G)$.
We say that $(X,S,T)$ is {\it automorphic} (or $S$ and $T$ are {\it automorphic}) if for any nets $u_{S,i}\in E(X,S)$ and $u_{T,i}\in E(X,T)$ with $\lim u_{S,i}=u_S$ and $\lim u_{T,i}=u_T$, we have that $\lim u_{S,i}u_{T,i}=u_Su_T$. Equivalently, $S$ and $T$ are automorphic if the map $E(X,S)\times E(X,T)\to E(X,G)$, $(u_S,u_T)\mapsto u_Su_T$ is continuous.

The following theorem characterizes the enveloping semigroup for systems with production extensions:

\begin{thm} \label{EnvEquic}
Let $(X,S,T)$ be a system with commuting transformations $S$ and $T$. Then $(X,S,T)$ has a product extension if and only if $S$ and $T$ are automorphic.
Particularly, $E(X,G)=E(X,S)E(X,T)\coloneqq\{u_{S}u_{T}\colon u_{S}\in E(X,S), u_{T}\in E(X,T)\}$, and $E(X,S)$ commutes with $E(X,T)$.
\end{thm}

\begin{proof}

First, we prove that the property of being automorphic is preserved under factor maps. Let $\pi\colon Y\to X$ be a factor map between the systems $(Y,S,T)$ and $(X,S,T)$ and suppose that $(Y,S,T)$ is automorphic. Suppose that $(X,S,T)$ is not automorphic. Then there exist nets $u_{S,i}\in E(X,S)$ and $u_{T,i}\in E(X,T)$ such that $u_{S,i}u_{T,i}$ does not converge to $u_Su_T$. Taking a subnet, we can assume that $u_{S,i}u_{T,i}$ converges to $u\in E(X,G)$. Let $\pi^{\ast}\colon E(Y,G)\to E(X,G)$ be the map induced by $\pi$ and let $v_{S,i}\in E(Y,S)$ and $v_{T,i}\in E(Y,T)$ be nets with $\pi^{\ast}(v_{S,i})=u_{S,i}$ and $\pi^{\ast}(v_{T,i})=u_{T,i}$. Assume without loss of generality that $v_{S,i}\to v_{S}$ and $v_{T,i}\to v_{T}$. Then $v_{S,i}v_{T,i}\to v_{S}v_{T}$. So $u_{S,i}u_{T,i}\to u_Su_T=u$, a contradiction.  On the other hand, since a product system is clearly automorphic, we get the first implication.

Now suppose that $S$ and $T$ are automorphic.

{\it Claim 1: $E(X,S)$ commutes with $E(X,T)$.}

Indeed, let $u_S \in E(X,S)$ and $u_T\in E(X,T)$ . Let $(n_i)$ be a net such that $S^{n_i} \to u_S$. Then $S^{n_i}u_T \to  u_S u_T$. On the other hand, since $S$ commutes  with $E(X,T)$ we have that $S^{n_i} u_T = u_T S^{n_i}$ for every $i$ and this converges to $u_Tu_S$ by the hypothesis of automorphy.

{\it Claim 2 : For any $x\in X$, $\bold{K}^{x}_{S,T}=\{(u_Sx,u_Tx,u_Su_Tx)\colon u_S \in E(X,S), ~ u_T\in E(X,T) \}$. }

We recall that $\bold{K}^{x}_{S,T}$ in invariant under $S\times\id\times S$ and $\id \times T \times T$. Since $\bold{K}^{x}_{S,T}$ is closed we have that is invariant under $u_S\times \id \times u_S$ and $\id \times u_T\times u_T$ for any $u_S\in E(X,S)$ and $u_T\in E(X,T)$. Hence $(u_S\times \id \times u_S)(\id \times u_T\times u_T)(x,x,x)=(u_Sx,u_Tx,u_Su_Tx)\in K^x_{S,T}$.

Conversely, let $(a,b,c)\in\bold{K}^{x}_{S,T}$. Let $(m_i)_{i\in \N}$ and $(n_i)_{\in \N}$ be sequences in $\Z$ such that
$S^{m_i} x \to a$, $T^{n_i} x \to b$ and $S^{m_i} T^{n_i} x \to c$. Replacing these sequences with finer filters, we can assume that $S^{m_i}\to u_S \in E(X,S)$ and $T^{n_i} \to u_T \in E(X,T)$. By the hypothesis of automorphy, $S^{m_i} T^{n_i} \to u_Su_T$ and thus $u_Su_Tx = c$ and
$(a, b, c) = (u_Sx, u_Tx, u_Su_Tx)$. The claim is proved.

Let $(a,b,c)$ and $(a,b,d) \in \bold{K}^{x}_{S,T}$. We can take $u_S,u_S'\in E(X,S)$ and $u_T,u_T'\in E(X,T)$ such that $(a,b,c)=(u_Sx,u_Tx,u_Su_Tx)$ and $(a,b,d)=(u_S'x,u_T'x,u_S'u_T'x)$. Since $E(X,S)$ and $E(X,T)$ commute we deduce that $c=u_Su_Tx=u_Sb=u_Su_T'x=u_T'u_Sx=u_T'a=u_T'u_S'x=d$.

Consequently, the last coordinate of $\bold{K}^{x}_{S,T}$ is a function of the first two ones. By Proposition \ref{FunctionCoordinate}, $(X,S,T)$ has a product extension.

\end{proof}

\subsection{Disjointness of systems with a product extension}
We recall the definition of disjointness:

\begin{defn}
Let $(X,G_0)$ and $(Y,G_0)$ be two dynamical systems. A \emph{joining} between $(X,G_0)$ and $(Y,G_0)$ is a closed subset $Z$ of $X\times Y$ which is invariant under the action $g\times g$ for all $g\in G_0$ and projects onto both factors. We say that $(X,G_0)$ and $(Y,G_0)$ are \emph{disjoint} if the only joining between them is their Cartesian product.
\end{defn}

\begin{defn}
  Let $(X,S,T)$ be a minimal system with commuting transformations $S$ and $T$. We say that a point $x\in X$ is {\it $S$-$T$ almost periodic} if $x$ is an almost periodic point of the systems $(X,S)$ and $(X,T)$. Equivalently, $x$ is $S$-$T$ almost periodic if  $(\mathcal{O}_S(x),S)$ and $(\mathcal{O}_T(x),T)$ are minimal systems. The system $(X,S,T)$ is {\it $S$-$T$ almost periodic} if every point $x\in X$ is $S$-$T$ almost periodic.
\end{defn}

\begin{rem}
 We remark that if $(K^{x}_{S,T},\widehat{S},\widehat{T})$ is minimal, then $x$ is $S$-$T$ is almost periodic. Consequently, if $(X,S,T)$ has a product extension we have that $(K^{x}_{S,T},\widehat{S},\widehat{T})$ is minimal for every $x\in X$ and then $(X,S,T)$ is $S$-$T$ almost periodic.
\end{rem}

The main theorem of this subsection is:

\begin{thm}\label{disjointn}
  Let $(X,S,T)$ be an $S$-$T$ almost periodic system. Then $(X,S)$ and $(X,T)$ are minimal and weak mixing if and only if $(X,S,T)$ is disjoint from all systems with product extension.
\end{thm}

We begin with a general lemma characterizing the relation of transitivity with the cube structure:

\begin{lem}\label{QTransitive}
Let $(X,T)$ be a topological dynamical system. Then $(X,T)$ is transitive if and only if $\Q_T(X)=X\times X$.
\end{lem}

\begin{proof}
 Let $x\in X$ be a transitive point. We have that $X\times X$ is the orbit closure of $(x,x)$ under $T\times T$ and $\id\times T$. Since $\Q_T(X)$ is invariant under these transformations we conclude that $\Q_T(X)=X\times X$.

 Conversely let $U$ and $V$ be two non-empty open subsets and let $x\in U$ and $y\in V$. Since $(x,y)\in \Q_T(X)$, there exist $x'\in X$ and $n\in \Z$ such that $(x',T^nx')\in U\times V$. This implies that $U\cap T^{-n}V\neq \emptyset$.
\end{proof}

We recall the following lemma (\cite{P}, page 1):

\begin{lem} \label{LemmaWeakMixing}
 Let $(X,T)$ be a topological dynamical system. Then $(X,T)$  weakly mixing if and only if for every two non-empty open sets $U$ and $V$ there exists $n\in \Z$ with  $U\cap T^{-n}U\neq \emptyset$ and $ U\cap T^{-n}V\neq \emptyset$
\end{lem}

The following lemma characterizes the weakly mixing property in terms of the cube structure:

\begin{lem}
 Let $(X,T)$ be a topological dynamical system. The following are equivalent:
   \begin{enumerate}
    \item $(X,T)$ is weakly mixing;
    \item $\Q_{T,T}(X)=X\times X\times X \times X$;
    \item $(x,x,x,y)\in \Q_{T,T}(X)$ for every $x,y\in X$.
   \end{enumerate}
\end{lem}

\begin{proof}

 $(1)\Rightarrow (2).$ Let suppose that $(X,T)$ is weakly mixing and let $x_0,x_1,x_2,x_3\in X$. Let $\e>0$ and for $i=0,1,2,3$ let $U_i$ be the open balls of radius $\e$ centered at $x_i$. Since $(X,T)$ is weak mixing there exists $n\in \Z$ such that $U_0\cap T^{-n}U_1\neq \emptyset$ and $U_2\cap T^{-n}U_3\neq \emptyset$. Since $(X,T)$ is transitive we can find a transitive point in $x'\in U_0\cap T^{-n} U_1$. Let $m\in \Z$ such that $T^mx'\in U_2\cap T^{-n}U_3$. Then $(x',T^nx',T^mx',T^{n+m}x')\in U_0\times U_1\times U_2\times U_3$ and this point belongs to $\Q_{T,T}(X)$. Since $\e$ is arbitrary we conclude that $(x_0,x_1,x_2,x_3)\in \Q_{T,T}(X)$.

 $(2)\Rightarrow (3).$ Clear.

 $(3)\Rightarrow (1).$ Let $U$ and $V$ be non-empty open sets and let $x\in U$ and $y\in V$. Since $(x,x,x,y)\in \Q_{T,T}(X)$, there exist $x'\in X$ and $n,m\in \Z$ such that $(x',T^nx',T^mx',T^{n+m}x')\in U\times U\times U\times V$. Then $x'\in U\cap T^{-n}U$ and $T^mx'\in U\cap T^{-n}V$ and therefore $U\cap T^{-n}U\neq \emptyset$ and $U\cap T^{-n}V\neq \emptyset$. By Lemma \ref{LemmaWeakMixing} we have that $(X,T)$ is weak mixing.
\end{proof}

\begin{rem}
When $(X,T)$ is minimal, a stronger results hold \cite{SY}, Subsection 3.5.
\end{rem}

The following is a well known result rephrased in our language:

\begin{prop}\label{mw}
  Let $(X,T)$ be a minimal system. Then $\mathcal{R}_{T,T}(X)=X\times X$ if and only if $(X,T)$ is weakly mixing.
\end{prop}

\begin{proof}
 If $(X,T)$ is minimal we have that $(x,y)\in \mathcal{R}_{T,T}(X)$ if and only if $(x,x,x,y)\in \Q_{T,T}(X)$ \cite{HKM}, \cite{SY}.
\end{proof}

\begin{rem}
 If $(X,T)$ is not minimal, it is not true that $\mathcal{R}_{T,T}(X)=X\times X$ implies that $(X,T)$ is weakly mixing. For instance, let consider the set $X\coloneqq \{1/n\colon n>1\} \cup \{1-1/n\colon n>2\}\cup \{0\}$ and let $T$ be the transformation defined by $T(0)=0$ and for $x\neq 0$, $T(x)$ is the number that follows $x$ to the right. If $x$ and $y$ are different from $0$, then $(x,x,x,y)\in \Q_{T,T}$ implies $x=y$ and thus $(X,T)$ is not weakly mixing. On the other hand, if $x$ and $y$ are different from 0, then there exists $n\in \Z$ with $y=T^nx$. Then $\lim\limits_{i\to \infty}(x,T^nx,T^ix,T^{n+i}x)=(x,y,0,0)$ meaning that $(x,y)\in \mathcal{R}_{T,T}(X)$. Since $\mathcal{R}_{T,T}(X)$ is closed we have that $\mathcal{R}_{T,T}(X)=X\times X$.

\end{rem}

\begin{lem}\label{trrp}
  Let $(X,S,T)$ be a minimal system with commuting transformations $S$ and $T$. If $S$ is transitive, then $\mathcal{R}_{T,T}(X)\subseteq\mathcal{R}_{S,T}(X)\subseteq\mathcal{R}_{S,S}(X)$.
\end{lem}
\begin{proof}
  Suppose $(x,y)\in\mathcal{R}_{S,T}(X)$. For $\e>0$, there exist $z\in X$, $n,m\in\mathbb{Z}$ such that $d(x,z)<\e$, $d(y,S^{n}z)<\e$ and $d(T^{m}z,S^{n}T^{m}z)<\epsilon$. Pick $0<\delta<\epsilon$ such that $d(x',y')<\delta$ implies
  $d(S^{n}x',S^{n}y')<\epsilon$ for all $x',y'\in X$. Since $S$ is transitive, there exist
  $z'\in X$, $r\in\mathbb{Z}$ such that $d(z,z')<\delta$ and $d(T^{m}z,S^{r}z')<\delta$. So $d(S^{n}z,S^{n}z')<\e$ and $d(S^{n}T^{m}z,S^{n+r}z')<\epsilon$. Thus $d(x,z')<2\epsilon$, ~$d(y,S^{n}z')<2\e$ and $d(S^{r}z',S^{r+n}z')<3\epsilon$.
 Since $\epsilon$ is arbitrary, $(x,y)\in\mathcal{R}_{S,S}(X)$.

  Suppose $(x,y)\in\mathcal{R}_{T,T}(X)$. Then there exists $a\in X$ such that for any $\epsilon>0$, there exists
  $z\in X, m,n\in\mathbb{Z}$ such that $d(x,z)$, ~$d(y,T^{m}z)$, ~$d(a,T^{n}z)$ and $d(a,T^{n+m}z)<\epsilon$. Pick $0<\delta<\epsilon$ such that $d(x',y')<\delta$ implies
  $d(T^{n}x',T^{n}y')<\epsilon$ for all $x',y'\in X$. Since $S$ is transitive, there exists
  $z'\in X$, $r\in\mathbb{Z}$ such that $d(z,z')<\delta$ and $d(T^{m}z,S^{r}z')<\delta$. So $d(T^{n}z,T^{n}z')<\e$ and $d(T^{n+m}z,T^{n}S^{r}z')<\epsilon$. Thus $d(x,z')<\e$,~ $d(y,S^{r}z')<\e$,~ $d(a,T^{n}z')<\e$, ~$d(a,T^{n}S^{r}z')<2\epsilon$.
Since $\epsilon$ is arbitrary, $(x,y,a,a)\in\Q_{S,T}(X)$. Similarly, $(x,b,y,b)\in\Q_{S,T}(X)$ for some $b\in X$. So $(x,y)\in\mathcal{R}_{S,T}(X)$.
\end{proof}

\begin{lem} \label{R_STWeakMixing}
  Let $(X,S,T)$ be a system with commuting transformations $S$ and $T$ such that both $S$ and $T$ are minimal. Then $\mathcal{R}_{S,T}(X)=X\times X$ if and only if both $(X,S)$ and $(X,T)$ are weakly mixing.
\end{lem}
\begin{proof}
  If both $(X,S)$ and $(X,T)$ are weakly mixing, then $\mathcal{R}_{S,S}(X)=X\times X$ and $T$ is transitive. By Lemma \ref{trrp}, $\mathcal{R}_{S,T}(X)=X\times X$.

  Now suppose that $\mathcal{R}_{S,T}(X)=X\times X$. For any $x,y\in X$, since $(x,y)\in\mathcal{R}_{S,T}(X)$, we may assume that $(x,a,y,a)\in \Q_{S,T}(X)$ for some $a\in X$. For any $\epsilon>0$, there exists
  $z\in X$, $n,m\in\mathbb{Z}$ such that $d(x,z)<\e$,~ $d(a,S^{n}z)<\e$,~ $d(y,T^{m}z)<\e$,~ $d(a,S^{n}T^{m}z)<\epsilon$. Pick $0<\delta<\epsilon$ such that $d(x',y')<\delta$ implies
  $d(S^{n}x',S^{n}y')<\epsilon$ for all $x',y'\in X$. Since $(z,T^{m}z)\in\mathcal{R}_{S,T}(X)$, there exist
  $z'\in X$, $r\in\mathbb{Z}$ such that $d(z,z')<\delta$,~ $d(T^{m}z,S^{r}z')<\delta$. So $d(S^{n}z,S^{n}z')<\e$,~ $d(S^{n}T^{m}z,S^{n+r}z')<\epsilon$. Thus $d(x,z')<2\e$,~ $d(a,S^{n}z')<2\e$,~ $d(y,S^{r}z')<2\e$ and $d(a,S^{n+r}z')<2\epsilon$. Since $\epsilon$ is arbitrary, $(x,y)\in\mathcal{R}_{S,S}(X)$. So $\mathcal{R}_{S,S}(X)=X\times X$ and since $S$ is minimal we have that $(X,S)$ is weakly mixing. Similarly, $(X,T)$ is weakly mixing.
\end{proof}

Shao and Ye proved ~\cite{SY} the following lemma in the case when $S=T$, but the same method works for the general case. So we omit the proof:
\begin{lem}\label{equvi}
Let $(X,S,T)$ be a system with commuting transformations $S$ and $T$ such that both $S$ and $T$ are minimal. Then the following are equivalent:
\begin{enumerate}
\item $(x,y)\in\mathcal{R}_{S,T}(X)$;

\item $(x,y,y,y)\in \bold{K}^{x}_{S,T}$;

\item $(x,x,y,x)\in \bold{K}^{x}_{S,T}$.
\end{enumerate}
\end{lem}
\begin{rem}
 We remark that a transformation is minimal if and only if it is both almost periodic and transitive.
\end{rem}

\begin{lem} \label{LiftingProperty2}
 Let $(X,S,T)$ be a system with commuting transformations $S$ and $T$ such that $(X,S)$ and $(X,T)$ are minimal and weak mixing. Let $(Y,S,T)$ be a minimal system with commuting transformations $S$ and $T$ such that $(Y,S,T)$ has a product extension. Let $Z\subset X\times Y$ be a closed subset of $X\times Y$ which is invariant under $\overline{S}=S\times S$ and $\overline{T}=T\times T$. Let $\pi\colon Z\to X$ be the natural factor map. For $x_{1},x_{2}\in X$, if there exists $y_{1}\in Y$ such that $z_{1}=(x_{1},y_{1})\in Z$ is a $S$-$T$ almost periodic point, then there exists $y\in Y$ such that $(x_{1},y),(x_{2},y)\in Z$.
\end{lem}

\begin{proof}
  By Lemma \ref{equvi}, $(x_{1},x_{2},x_{2},x_{2})\in \bold{K}^{x_{1}}_{S,T}$. So there exists a sequence $(F_{i})_{i\in \N}\in\F_{S,T}$
   such that
   \begin{equation}\nonumber
     \begin{split}
       \lim_{i\to\infty}F_{i}(x_{1},x_{1},x_{1},x_{1})=(x_{1},x_{2},x_{2},x_{2}).
     \end{split}
   \end{equation}
   Recall that $z_{1}=(x_{1},y_{1})\in\pi^{-1}(x_{1})$ . Without loss of generality, we assume that
      \begin{equation}\label{e1}
     \begin{split}
       &\lim_{i\to\infty}F_{i}(y_{1},y_{1},y_{1},y_{1})=(y_{1},y_{2},y_{3},y_{4});
       \\&\lim_{i\to\infty}\overline{F}_{i}(z_{1},z_{1},z_{1},z_{1})=(z_{1},z_{2},z_{3},z_{4}),
     \end{split}
   \end{equation}
   where $\overline{F}_{i}=F_{i}\times F_{i}$ and $z_{2}=(x_{2},y_{2}),z_{3}=(x_{2},y_{3}),z_{4}=(x_{2},y_{4})$ are points in $Z$. Since $(x_1,y_1)$ is $S$-$T$ almost periodic, there exists a sequence of integers $(n_{i})_{i\in \N}$
   such that $\lim_{i\to\infty}\overline{S}^{n_{i}}z_{2}=z_{1}$. We can assume that $\lim_{i\to\infty}\overline{S}^{n_{i}}z_{4}=z'_{4}=(x_1,y')\in Z$. Then
   \begin{equation}\label{e2}
     \begin{split}
       \lim_{i\to\infty}(\id\times\overline{S}\times \id\times\overline{S})^{n_{i}}(z_{1},z_{2},z_{3},z_{4})=(z_{1},z_{1},z_{3},z'_{4}).
     \end{split}
   \end{equation}
  This implies that $(y_1,y_1,y_3,y')\in \Q_{S,T}(Y)$ by Theorem \ref{R-ProductExtension} since $\mathcal{R}_{S}(Y)=\Delta_X$ we have that $y'=y_3$. Therefore $z_4'=(x_1,y_3)$ and $z_{3}=(x_2,y_3)$ belong to $Z$.
\end{proof}

We are now finally able to prove the main theorem of this subsection:

\begin{proof}[Proof of Theorem \ref{disjointn}]
   Let $(X,S,T)$ be a system such that $(X,S)$ and $(X,T)$ are minimal weak mixing and let $(Y,S,T)$ be a system with a product extension.
  Suppose $Z\subseteq X\times Y$ is closed and invariant under $\overline{S}=S\times S,\overline{T}=T\times T$. We have to show that $Z=X\times Y$. Let
  $\mathcal {W}=\{Z\subseteq X\times Y\colon Z$ is closed invariant under $\overline{S}=S\times S,\overline{T}=T\times T \}$ with order $Z\leq Z'$ if and only if $Z'\subset Z$. Let $\{Z_{i}\}_{i\in I}$ be a totally ordered subset of $\mathcal {W}$ and denote $Z_{0}=\cap_{i\in I} Z_{i}$. It is easy to see that $Z_{0}\in\mathcal {W}$. By Zorn's Lemma, we can assume $Z$ contains no proper closed invariant subset.

  For any $x\in X$, denote $F_{x}=\{y\in Y\colon(x,y)\in Z\}$. Then $F_{x}\subseteq Y$ is a closed set of $Y$.

  For any $g\in G$, let $Z_{g}=\{(x,y)\in X\times Y\colon y\in(F_{x}\cap gF_{x})\}$. Then $Z_{g}\subseteq Z$ is closed invariant. Since $Z$ contains no proper invariant subset, either $Z_{g}=\emptyset$ or $Z_{g}=Z$.
   Denote $U=\{x\in X\colon \exists y\in Y, (x,y)$ is an almost periodic point of $Z\}$. For any $x_{0}\in U$, suppose $z_{0}=(x_{0},y_{0})\in Z$ is an $S$-$T$ almost periodic point.
   For any $g\in G, (x_{0},gx_{0})\in\mathcal{R}_{S,T}(X)$. By Proposition \ref{LiftingProperty2}, there exists $y\in Y$ such that $(x_{0},y),(gx_{0},y)\in Z$. So $F_{x_{0}}\cap F_{gx_{0}}=F_{x_{0}}\cap gF_{x_{0}}\neq\emptyset$. Therefore $Z_{g}\neq\emptyset$. So $Z_{g}=Z$ for all $g\in G$.
  Thus $F_{x}=gF_{x}$ for every $x\in U$. Since $g$ is arbitrary, $F_{x}$ is closed invariant under $G$ for every $x\in U$. Since $(Y,G)$ is minimal, and $F_{x}\neq \emptyset$ we get that $F_{x}=Y$ for all $x\in U$.

  It suffices to show that $U=X$. Fix $x\in X$. Since $x$ is $S$-$T$-almost periodic, there exist minimal idempotents $u_S\in E(X,S)$ and $u_T\in E(X,T)$ such that $u_Sx=x=u_Tx$. These idempotents can be lifted to minimal idempotents in $E(Z,S)$ and $E(Z,T)$ which can be projected onto minimal idempotents in $E(Y,S)$ and $E(Y,T)$. We also denote these idempotents by $u_S$ and $u_T$. By Theorem \ref{EnvEquic}, these idempotents commute in $E(Y,G)$. So for $y\in Y$ such that $(x,y)\in Z$, we have that $u_Su_T(x,y)=(x,u_Su_Ty)\in Z$, and $u_S(x,u_Su_Ty)=(x,u_Su_Ty)$, $u_T(x,u_Su_Tx)=(x,u_Su_Tx)$. This means that the point $(x,u_Su_Tx)\in Z$ is $S$-$T$-almost periodic. Hence $U=X$ and therefore $Z=X\times Y$.

  Conversely, let $(X,S,T)$ be a system disjoint from systems with product extension.  Let $U$ and $V$ be non-empty open subsets of $X$ and let $x\in U$ and $y\in V$. Since $X$ is $S$-$T$ almost periodic, we have that $(\mathcal{O}_S(x),S)$ and  $(\mathcal{O}_T(x),T)$ are minimal systems. By hypothesis, $(X,S,T)$ is disjoint from $(\mathcal{O}_S(x)\times \mathcal{O}_T(x),S\times \id , \id \times T)$. Since $(x,(x,x))$ and $(y,(x,x))$ belong to $X\times (\mathcal{O}_S(x)\times \mathcal{O}_T(x))$, we have that there exist sequences $(n_i)_{i\in \N}$ and $(m_i)_{i\in \N}$ in $\Z$ such that $(S^{n_i}T^{m_i}x,(S^{n_i}x,T^{m_i}x))\to (y,(x,x))$. Particularly $(x,S^{m_i}x,T^{m_i}x,S^{n_i}T^{m_i}x)\in \Q_{S,T}(X)$ and this point converges to $(x,x,x,y)\in \Q_{S,T}(X)$. This implies that $(x,y)\in \Q_{S}(X)$, $(x,y)\in \Q_T(X)$ and $(x,y)\in \mathcal{R}_{S,T}(X)$ and since $x$ and $y$ are arbitrary we deduce that $\Q_S(X)=\Q_T(X)=\mathcal{R}_{S,T}(X)=X\times X$. By Lemma \ref{QTransitive} we deduce that $S$ and $T$ are transitive and since $(X,S,T)$ is $S$-$T$ almost periodic we deduce that $S$ and $T$ are minimal. By Lemma \ref{R_STWeakMixing} we deduce that $(X,S)$ and $(X,T)$ are minimal and weak mixing.

\end{proof}

\subsection{Recurrence in systems with a product extension}
We define of sets of return times in our setting:

\begin{defn}

Let $(X,S,T)$ be a minimal distal system with commuting transformations $S$ and $T$, and let $x\in X$. Let $x\in X$ and $U$ be an open neighborhood of $x$. We define the {\it set of return times} $N_{S,T}(x,U)=\{(n,m)\in \Z^2\colon S^nT^mx\in U\}$, $N_S(x,U)=\{n\in \Z\colon S^nx\in U\}$ and $N_T(x,U)=\{m\in \Z\colon T^mx\in U\}$.

A subset $A$ of $\Z$ is a {\it set of return times for a distal system} if there exists a distal system $(X,S)$, an open subset $U$ of $X$ and $x\in U$ such that $N_S(x,U)\subseteq A$.

A subset $A$ of $\Z$ is a $Bohr_0$ set is here exists an equicontinuous system $(X,S)$, an open subset $U$ of $X$ and $x\in U$ such that $N_S(x,U)\subseteq A$.

\end{defn}

\begin{rem}
We remark that we can characterize $\Z^2$ sets of return times of distal systems with a product extension: they contain the Cartesian product of sets of return times for distal systems.
Let $(X,S,T)$ be a minimal distal system with a product extension $(Y\times W,\sigma\times \id,\id\times \tau)$, and let $U$ be an open subset of $X$ and $x\in U$. By Theorem \ref{ThmDistalCas} we can assume that the product extension is also distal.  Let $\pi$ denote a factor map from $Y\times W \to X$. Let $(y,w)\in Y\times W$ such that $\pi(y,w)=x$ and let $U_Y$ and $U_W$ be neighborhoods of $y$ and $w$ such that $\pi(U_Y\times U_W)\subseteq U$. Then we have that that $N_{\sigma}(y,U_Y)\times N_{\tau}(w,U_W)\subseteq N_{S,T}(x,U)$.

Conversely, let $(Y,\sigma)$ and $(W,\tau)$ be minimal distal systems. Let $U_Y$ and $U_W$ be non-empty open sets in $Y$ and $W$ and let $y\in U_Y$ and $w\in U_W$. Then $N_{\sigma}(y,U_Y)\times N_{\tau}(w,U_W)$ coincides with $N_{\sigma\times\id,\id\times \tau}((y,w),U_Y\times U_W)$.

\end{rem}

Denote by $\mathcal{B}_{S,T}$ the family generated by Cartesian products of sets of return times for a distal system. Equivalently $\mathcal{B}_{S,T}$ is the family generated by sets of return times arising from minimal distal systems with a product extension.

Denote by $\mathcal{B}^{*}_{S,T}$ the family of sets which have non-empty intersection with every set in $\mathcal{B}_{S,T}$.

\begin{lem} \label{SuperLifting}
 Let $(X,S,T)$ be a minimal distal system with commuting transformations $S$ and $T$, and suppose $(x,y)\in \mathcal{R}_{S,T}(X)$. Let $(Z,S,T)$ be a minimal distal system with $\mathcal{R}_{S,T}(Z)=\Delta_Z$
 and let $J$ be a closed subset of $X\times Z$, invariant under $T\times T$ and $S\times S$. Then for $z_0\in Z$ we have $(x,z_0)\in J$ if and only if $(y,z_0)\in J$.
\end{lem}

\begin{proof}
We adapt the proof of Theorem 3.5 \cite{HSY} to our context. Let $W=Z^{Z}$ and $S^{Z},T^{Z}\colon W\rightarrow W$ be such that for any $\omega\in W$, $(S^{Z}\omega)(z)=S(\omega(z))$, $(T^{Z}\omega)(z)=T(\omega(z))$, $z\in Z$. Let $\omega^{*}\in W$ be the point satisfying $\omega(z)=z$ for all $z\in Z$ and let $Z_{\infty}=\mathcal{O}_{G^{Z}}(\omega^{*})$, where $G^{Z}$ is the group generated by $S^{Z}$ and $T^{Z}$. It is easy to verify that $Z_{\infty}$ is minimal distal. So for any $\omega\in Z_{\infty}$, there exists $p\in E(Z,G)$ such that $\omega(z)=p\omega^{*}(z)=p(z)$ for any $z\in Z$. Since $(Z,S,T)$ is minimal and distal, $E(Z,G)$ is a group (see \cite{Aus}, Chapter 5). So $p\colon Z\rightarrow Z$ is surjective. Thus there exists $z_{\omega}\in Z$ such that $\omega(z_{\omega})=z_{0}$.

  Take a minimal subsystem $(A,S\times S^{Z},T\times T^{Z})$ of the product system $(X\times Z_{\infty},S\times S^{Z},T\times T^{Z})$. Let $\pi_{X}\colon (A,S\times S^{Z},T\times T^{Z})\rightarrow (X,S,T)$ be the natural coordinate projection. Then $\pi_{X}$ is a factor map between two distal minimal systems. By Proposition \ref{LiftingProperty}, there exists $\omega^{1},\omega^{2}\in W$ such that $((x,\omega^{1}),(y,\omega^{2}))\in \mathcal{R}_{S',T'}(A)$, where $S'=S\times S^{Z}, T'=T\times T^{Z}$.

  Let $z_{1}\in Z$ be such that $\omega^{1}(z_{1})=z_{0}$. Denote $\pi\colon A\rightarrow X\times Z$, $\pi(u,\omega)=(u,\omega(z_{1}))$ for $(u,\omega)\in A$, $u\in X$ and $\omega\in W$. Consider the projection $B=\pi(A)$. Then $(B,S\times S,T\times T)$ is a minimal distal subsystem of $(X\times Z, S\times S,T\times T)$ and since $\pi(x_0,\omega^{1})=(x,z_{0})\in B$ we have that $J$ contains $B$. Suppose that $\pi(x,\omega^{2})=(x,z_{2})$. Then $((x,z_{0}),(y,z_{2}))\in \mathcal{R}_{S\times S,T\times T}(B)$ and thus $(z_{0},z_{2})\in \mathcal{R}_{S,T}(Z)$. Since $\mathcal{R}_{S,T}(Z)=\Delta_{Z\times Z}$ we have that $z_{0}=z_{2}$ and thus $(y,z_0)\in B\subseteq J$
\end{proof}

\begin{thm} \label{RecRPST}
  Let $(X,S,T)$ be a minimal distal system with commuting transformations $S$ and $T$. Then for $x,y\in X$, $(x,y)\in \mathcal{R}_{S,T}(X)$ if and only if $N_{S,T}(x,U)\in\mathcal{B}^{*}_{S,T}$ for any open neighborhood $U$ of $y$.
\end{thm}
\begin{proof}
  Suppose $N(x,U)\in\mathcal{B}^{*}_{S,T}$ for any open neighborhood $U$ of $y$. Since $X$ is distal, $\mathcal{R}_{S,T}(X)$ is an equivalence relation. Let $\pi$ be the projection map $\pi\colon X\rightarrow Y\coloneqq X/\mathcal{R}_{S,T}(X)$. By Theorem \ref{ThmDistalCas} we have that $\mathcal{R}_{S,T}(Y)=\Delta_{Y}$. Since $(X,S,T)$ is distal, the factor map $\pi$ is open and $\pi(U)$ is an open neighborhood of $\pi(x)$. Particularly $N_{S,T}(x,U)\subseteq N_{S,T}(\pi(x),\pi(U))$. Let $V$ be an open neighborhood of $\pi(x)$. By hypothesis we have that $N_{S,T}(x,U)\cap N_{S,T}(\pi(x),\pi(U))\neq\emptyset$ which implies that $N_{S,T}(\pi(x),\pi(U))\cap N_{S,T}(\pi(x),V)\neq \emptyset$. Particularly $\pi(U)\cap V\neq \emptyset$. Since this holds for every $V$ we have that $\pi(x)\in \overline{\pi(U)}=\pi(\overline{U})$. Since this holds for every $U$ we conclude that $\pi(x)=\pi(y)$. This shows that $(x,y)\in \mathcal{R}_{S,T}(X)$.

  Conversely, suppose that $(x,y)\in \mathcal{R}_{S,T}(X)$, let $U$ be an open neighborhood of $y$ and let $A$ be a $\mathcal{B}^{*}_{S,T}$ set. Then, there exists a minimal distal system $(Z,S,T)$ with $\mathcal{R}_{S,T}(Z)=\Delta_Z$, an open set $V\subseteq Z$ and $z_0\in V$ such that $N_{S,T}(z_0,V)\subseteq A$. Let $J$ be orbit closure of $(x,z_0)$ under $S\times S$ and $T\times T$. By distality we have that $(J,S\times S,T\times T)$ is a minimal system and $(x,z_0)\in J$. By Lemma \ref{SuperLifting} we have that $(y,z_0)\in J$ and particularly, there exist sequences $(n_i)_{i\in \N}$ and $(m_i)_{i\in \N}$ in $\Z$ such that $(S^{n_i}T^{m_i}x,S^{n_i}T^{m_i}z_0)\to (y,z_0)$. This implies that $N_{S,T}(x,U)\cap N_{S,T}(z_0,V)\neq \emptyset$ and the proof is finished.
\end{proof}

\begin{cor} \label{ReturnTProduct}
 Let $(X,S,T)$ be a minimal distal system with commuting transformations $S$ and $T$. Then $(X,S,T)$ has a product extension if and only if for every $x\in X$ and every open neighborhood $U$ of $x$, $N_{S,T}(x,U)$ contains the product of two set of return times for a distal system.
\end{cor}

\begin{proof}
We prove the non-trivial implication. Let suppose that there exists $(x,y)\in \mathcal{R}_{S,T}(X)\setminus \Delta_X$ and let $U,V$ be open neighborhoods of $x$ and $y$ respectively such that $U\cap V=\emptyset$.  By assumption $N_{S,T}(x,U)$ is a $\mathcal{B}_{S,T}$ set, and by Theorem \ref{RecRPST} $N_{S,T}(x,V)$ has nonempty intersection with $N_{S,T}(x,U)$. This implies that $U\cap V\neq \emptyset$, a contradiction. We conclude that $\mathcal{R}_{S,T}(X)=\Delta_X$ and therefore $(X,S,T)$ has a product extension.
\end{proof}

Specially, when $S=T$ we get
\begin{cor} \label{EquicontinuousSum}
 Let $(X,T)$ be a minimal distal system. Then $(X,T)$ is equicontinuous if and only if for every $x\in X$ and every open neighborhood $U$ of $x$, $N_{T}(x,U)$ contains the sum of two sets of return times for distal systems.
\end{cor}

\begin{proof}
Suppose $(X,T)$ is equicontinuous, then the system $(X,T,T)$ with commuting transformations $T$ and $T$ has a product extension. So for every $x\in X$ and every open neighborhood $U$ of $x$, we have that $N_{T,T}(x,U)$ contains a product of two sets $A$ and $B$. In terms of the one dimensional dynamics, this means that $N_{T}(x,U)$ contains $A+B$.

Conversely, if $N_{T}(x,U)$ contains the sum of two sets of return times for distal systems $A$ and $B$, we have that $N_{T,T}(x,U)$ contains the set $A\times B$. By Corollary \ref{ReturnTProduct}, $(X,T,T)$ has a product extension and by Corollary \ref{EquicontinuousAndProduct} $(X,T)$ is an equicontinuous system.
\end{proof}

\begin{ques}
  A natural question arising from Corollary \ref{EquicontinuousSum} is the following: is the sum of two set of return times for a distal system a $Bohr_0$ set?
  \end{ques}

\subsection{Complexity for systems with a product extension} In this subsection, we study the complexity of a distal system with a product extension. We start recalling some classical definitions.

Let $(X,G_0)$ be a topological dynamical system. A finite cover $\mathcal{C}=(C_1,\ldots,C_d)$ is a finite collection of subsets of $X$ whose union is all $X$. We say that $\mathcal{C}$ is an {\it open cover} if every $C_i\in \mathcal{C}$ is an open set. Given two open covers $\mathcal{C}=(C_1,\ldots,C_d)$ and $\mathcal{D}=(D_1,\ldots, D_k)$ their {\it refinement} is the cover $\mathcal{C}\vee \mathcal{D}=(C_i\cap D_j\colon i=1,\ldots,d \quad j=1,\ldots, k)$. A cover $\mathcal{C}$ is {\it finer} than $\mathcal{D}$ if every element of $\mathcal{C}$ is contained in an element of $\mathcal{D}$. We let $\mathcal{D}\preceq \mathcal{C}$ denote this property.

We recall that if $(X,S,T)$ is a minimal distal system with commuting transformations $S$ and $T$ then $\Q_S(X)$, $\Q_T(X)$ and $\mathcal{R}_{S,T}(X)$ are equivalence relations.

Let $(X,S,T)$ be a minimal distal system with commuting transformations $S$ and $T$, and let $\pi_{S}$ be the factor map $\pi_{S}\colon X\rightarrow X/\Q_{S}(X)$. Denote $I_{S}=\{\pi_{S}^{-1}y\colon y\in X/\Q_{S}(X)\}$ the {\it set of fibers of $\pi_S$}.

Given a system $(X,S,T)$ with commuting transformations $S$ and $T$, and given a finite cover $\mathcal{C}$, denote $\mathcal{C}_{0}^{T,n}=\bigvee_{i=0}^{n}T^{-i}\mathcal{C}$. For any cover $\mathcal{C}$ and any closed $Y\subset X$, let $r(\mathcal{C},Y)$ be the minimal number of elements in $\mathcal{C}$ needed to cover the set $Y$. We remark that $\mathcal{D}\preceq \mathcal{C}$ implies that $r(\mathcal{D},Y)\leq r(\mathcal{C},Y)$.

\begin{defn}
Let $\mathcal{C}$ be a finite cover of $X$. We define the $S$-$T$ complexity of $\mathcal{C}$ to be the non-decreasing function
$$c_{S,T}(\mathcal{C},n)=\max_{Y\in I_{S}}r(\mathcal{C}_{0}^{T,n},Y).$$
\end{defn}

\begin{prop} \label{ComplexityRPST}
  Let $(X,S,T)$ be a distal system with commuting transformations $S$ and $T$. Then $(X,S,T)$ has a product extension if and only if $c_{S,T}(\mathcal{C},n)$ is bounded for any open cover $\mathcal{C}$.
 
\end{prop}
\begin{proof}
Suppose first that $\mathcal{R}_{S,T}(X)=\Delta_{X}$. Since $\Q_{S}(X)$ is an equivalence relation, by Proposition \ref{Eqextension}, we have that $\pi_{S}\colon (X,T)\rightarrow (X/\Q_{S}(X),T)$ is an equicontinuous extension.
Let $\epsilon>0$ be the Lebesgue number of the finite open cover $\mathcal{C}$, i.e. any open ball $B$ with radius $\epsilon$ is contained in at least one element of $\mathcal{C}$. Then there exists $0<\delta<\epsilon$ such that $d(x,y)<\delta$, $\pi_{S}(x)=\pi_{S}(y)$ implies that $d(T^{n}x,T^{n}y)<\epsilon$ for all $n\in\mathbb{Z}$. For any $Y\in I_{S}$, by compactness, let $x_{1},\dots,x_{k}\in Y$ be such that $Y\subset\bigcup_{i=1}^{k}B(x_{i},\delta)$. Then
$T^{j}(B(x_{i},\delta)\cap Y)\subset B(T^{j}x_{i},\epsilon)\cap Y\subset B(T^{j}x_{i},\epsilon)$ for any $j\in\mathbb{N}$ (since $\Q_{S}(X)$ is invariant under $T\times T$). Let $U_{i,j}$ be an element of $\mathcal{C}$ containing $B(T^{j}x_{i},\epsilon)$. Then $T^{j}(B(x_{i},\delta)\cap Y)\subset U_{i,j}$. So
$B(x_{i},\delta)\cap Y\subset\bigcap_{j=0}^{n}T^{-j}U_{i,j}$. Thus $\{\bigcap_{j=0}^{n}T^{-j}U_{i,j}\colon 1\leq i\leq k\}$ is a subset of $\mathcal{C}_{0}^{T,n}$ covering $Y$ with cardinality $k$. Therefore $r(\mathcal{C}_{0}^{T,n},Y)$ is bounded by the quantity of balls of radius $\delta$ needed to cover $Y$.

Suppose that $c_{S,T}(\mathcal{C},n)$ is not bounded. For $Y, Y' \in I_{S}$, let $d_H(Y,Y)$ be the Hausdorff distance between $Y$ and $Y'$. Since the factor map $X\to X/\Q_S$ is open, for any $\epsilon'>0$, there exists $\delta'>0$ such that if $y,y'\in X/\Q_S$ and $d(y,y')<\delta'$, then $d_H(\pi^{-1}y,\pi^{-1}y')<\epsilon'$.

Let $y\in Y$ and let $\mathcal{C}'\subseteq \mathcal{C}$ be a subcover of $Y=\pi^{-1}(y)$. Let $\e'>0$ be such that if $d(x,Y)<\e'$, then $x$ is covered by $\mathcal{C}'$ . We can find $\delta'>0$ such that if $d(y,y')<\delta'$, then $d_H(\pi^{-1}y,\pi^{-1}y')<\epsilon'$. Thus $\mathcal{C}'$ is also an open covering of $Y'=\pi^{-1}(y')$.

If $\pi^{-1}y\subset\bigcup_{i=1}^{k}B(x_{i},\delta)$, then there exists $\delta'>0$ such that $d(y,y')<\d'$ implies that $\pi^{-1}y' \subset \bigcup_{i=1}^{k}B(x_{i},\delta)$. If $c_{S,T}(\mathcal{C},n)$ is not bounded, there exists $y_i \in Y$ such that $\pi^{-1}(y_i)$ can not be covered by $i$ balls of radius $\delta>0$. We assume with out loss of generality that $y_i\to y$ (by taking a subsequence). Since $\pi^{-1}y$ can be covered by a finite number $K$ of balls of radius $\delta$, we get that for large enough $i$, $\pi^{-1}y_i$ can also be covered by $K$ balls of radius $\delta$, a contradiction. Therefore $c_{S,T}(\mathcal{C},n)$ is bounded.

Conversely, let suppose that $c_{S,T}(\mathcal{C},n)$ is bounded for every open cover $\mathcal{C}$ and suppose that $\mathcal{R}_{S,T}(X)\neq\Delta_{X}$. We remark that if $\mathcal{C}$ is an open cover and $Y\in I_{S}$ then $$r(\mathcal{C}_{-n}^{T,n},Y)\coloneqq r(\bigvee_{i=-n}^{n}T^{-i}\mathcal{C},Y)=r(T^n\bigvee_{i=0}^{2n}T^{-i}\mathcal{C},T^nT^{-n}Y)=r(\bigvee_{i=0}^{2n}T^{-i}\mathcal{C},T^{-n}Y).$$ Since $T$ commutes with $S$ we have that $T^{-n}Y\in I_{S}$ and thus the condition that $c_{S,T}(\mathcal{C},n)$ is bounded implies that $r(\bigvee_{i=-n}^{n}T^{-i}\mathcal{C},Y)$ is bounded for any $Y\in I_{S}$.

Since $\mathcal{R}_{S,T}(X)\neq\Delta_{X}$ by Proposition \ref{Eqextension}, there exist $\epsilon>0$ and $x\in X$ such that for any $\delta>0$, one can find $y\in X$ and $k\in\mathbb{Z}$ such that $d(x,y)<\delta$, $\pi_{S}(x)=\pi_{S}(y)$ and $d(T^{k}x,T^{k}y)>\epsilon$.  Pick any $Y\in I_{S}$ and let $\mathcal{C'}$ be a finite cover of open balls with radius $\epsilon/4$. Let $\mathcal{C}$ be the finite covering made up of the closures of the elements of $\mathcal{C'}$. Since $\mathcal{C}\prec \mathcal{C'}$ we have that $r(\mathcal{C}_{-n}^{T,n},Y)$ is also bounded.

By a similar argument of Lemma 2.1 of \cite{BHM}, there exist closed sets $X_{1},\dots,X_{c}\subset X$ such that $Y\subset\bigcup_{i=1}^{c}X_{i}$, where each $X_{i}$ can be written as $X_{i}=\bigcap_{j=-\infty}^{\infty}T^{-j}U_{i,j}$, with $U_{i,j}\in\mathcal{C}$. Then $y,z\in X_{i}$ implies that $d(T^{j}y,T^{j}z)<\epsilon/2$ for any $j\in\mathbb{Z}$.

Let $(\delta_{n})_{n\in \N}$ be a sequence of positive numbers such that $\lim_{n\rightarrow\infty}\delta_{n}=0$. For any $n\in \N$ we can find $y_{n}\in X$ and $k_{n}\in\mathbb{Z}$ with $d(x,y_{n})<\delta_n$, $\pi_{S}(x)=\pi_{S}(y_{n})$ and $d(T^{k_{n}}x,T^{k_{n}}y_{n})>\epsilon$. By taking a subsequence, we may assume that all $y_{n}$ belong to the same set $X_{i}$. Since $X_{i}$ is closed, $x\in X_{i}$. Thus $d(T^{j}x,T^{j}y_{n})<\epsilon/2$ for any $j,n\in\mathbb{N}$, a contradiction.

\end{proof}

\appendix

\section{General facts about the enveloping semigroup}

Let $(X,G_0)$ be a topological dynamical system. The {\it enveloping semigroup} (or {\it Ellis semigroup}) $E(X,G_0)$ of $(X,G_0)$ is
the closure in $X^X$ of the set $\{g\colon g\in G_0\}$ endowed with the product topology. For an enveloping semigroup $E(X,G_0)$, the applications $E(X,G_0)\to E(X,G_0)$, $p\mapsto pq$ and $p\mapsto gp$ are continuous for all $q\in E(X,G_0)$ and $g\in G_0$. We have that $(E(X,G_0),G_0)$ is dynamical system and if $G_0$ is abelian, $G_0$ is included in the center of $E(X,G_0)$. It is worth noting that $E(X,G_0)$ is usually not metrizable.

If $\pi\colon Y \to X$ is a factor map between the topological dynamical systems $(Y,G_0)$ and $(X,G_0)$, then $\pi$ induces a unique factor map $\pi^{\ast}\colon E(Y,G_0)\to E(X,G_0)$ that satisfies $\pi^{\ast}(u)y=\pi(uy)$ for every $u \in E(Y,G_0)$ and $y\in Y$.

This notion was introduced by Ellis \cite{Ellis} and allows the translation of algebraic properties into dynamical ones and vice versa.

We say that $u\in E(X,G_0)$ is an \emph{idempotent} if $u^2=u$. By the Ellis-Nakamura Theorem, any closed subsemigroup $H\subseteq E(X,G_0)$ admits an idempotent. A \emph{left ideal} $I\subseteq E(X,G_0)$ is a non-empty subset such that $E(X,G_0)I\subseteq I$. An ideal is \emph{minimal} if it contains no proper ideals. An idempotent $u$ is \emph{minimal} if $u$ belongs to some minimal ideal $I\subseteq E(X,G_0)$.

We summarize some results that connect algebraic and dynamical properties:

\begin{thm} \label{Enveloping1}
 Let $(X,G_0)$ be a topological dynamical system and let $E(X,G_0)$ be its enveloping semigroup. Then

 \begin{enumerate}
  \item An ideal $I\subseteq E(X,G_0)$ is minimal if and only if $(I,G_0)$ is a minimal system. Particularly, minimal ideals always exist;
  \item An idempotent $u\in E(X,G_0)$ is minimal if and only if $(\mathcal{O}_{G_0}(u),G_0)$ is a minimal system;
  \item An idempotent $u\in E(X,G_0)$ is minimal if $vu=v$ for some $v\in E(X,G_0)$ implies that $uv=u$;
  \item Let $x\in X$. Then $(\mathcal{O}_{G_0}(x),G_0)$ is a minimal system if and only if there exists a minimal idempotent $u\in E(X,G_0)$ with $ux=x$.
 \end{enumerate}

\end{thm}

\begin{thm} \label{Enveloping2}

Let $(X,G_0)$ be a topological dynamical system. Then
\begin{enumerate}
 \item $(x,y) \in P(X)$ if and only if there exists $u\in E(X,G_0)$ with $ux=uy$;
 \item Let $x\in X$ and let $u\in E(X,G_0)$ be an idempotent. Then $(x,ux)\in P(X)$;
 \item Let $x\in X$. Then there exists $y\in X$ such that $(x,y)\in P(X)$ and $(\mathcal{O}_{G_0}(y),G)$ is minimal.
\item If $(X,G_0)$ is minimal, $(x,y)\in P(X)$ if and only if there exists $u\in E(X,G_0)$ a minimal idempotent such that $y=ux$.

\end{enumerate}

\end{thm}

\begin{prop} \label{liftingIdempotent}
Let $(Y,G_0)$ and $(X,G_0)$ be topological dynamical systems and let $\pi\colon Y\to X$ be a factor map. If $u\in E(X,G_0)$ is a minimal idempotent, then there exists a minimal idempotent $v\in E(Y,G_0)$ such that $\pi^{\ast}(v)=u$.
\end{prop}

\begin{proof}
 If $u\in E(X,G_0)$ is a minimal idempotent, let $v'\in E(X,G_0)$ with $\pi^{\ast}(v')=u$. Then $\pi^{\ast}(\mathcal{O}_{G_0}(v'))=\mathcal{O}_{G_0}(u)$. Let $J\subseteq \mathcal{O}_{G_0}(v')$ be a minimal subsystem. Since $(\mathcal{O}_{G_0}(u),G)$ is minimal, we have that $\pi^{\ast}(J)=\mathcal{O}_{G_0}(u)$. Let $\phi$ be the restriction of $\pi^{\ast}$ to $J$. Since $u$ is idempotent, we have that $\phi^{-1}(u)$ is a closed subsemigroup of $E(Y,G_0)$. By the Ellis-Nakamura Theorem, we can find an idempotent $v\in \phi^{-1}(u)$. Since $v$ belongs to $J$ we have that $v$ is a minimal idempotent.

\end{proof}

\section*{Acknowledgments} We thank our advisors Bernard Host, Bryna Kra and Alejandro Maass for introducing us to the subject and for their help and useful discussions.

\end{document}